\documentclass{amsart}
\usepackage{amsmath,amssymb, amsthm, tikz}
\usepackage{pgfkeys}
\usepackage{graphics}
\usepackage{enumitem}
\newtheorem{lemm}{Lemma}[section]
\newtheorem{theorem}{Theorem}[section]
\newtheorem{proposition}{Proposition}[section]
\newtheorem{lemma}{Lemma}[section]
\newtheorem{cor}{Corollary}[section]

\newtheorem{propos}[lemm]{Proposition}

\newtheorem*{thm*}{ Theorem}
\newenvironment{defi}{\medskip\noindent{\sc
Definition}. }{\goodbreak\medskip}

\newenvironment{nota}{\medskip\noindent{\sc
Notation}.}{\goodbreak\medskip}
\newenvironment{remk}{\noindent{\sc
Remark}. }{\goodbreak\vskip10pt}
\newenvironment{remks}{\noindent{\sc
Remarks}. }{\goodbreak\vskip10pt}
\newenvironment{notas}{\medskip\noindent{\sc
Notations}. }{\goodbreak\medskip}

\def\cal{\mathcal}

\def\cF{{\mathcal F}}

\def\ct{{\mathcal T}}
\def\cw{{\mathcal W}}

\def\cg{{\mathcal G}}
\def\ca{{\mathcal A}}
\def\ch{{\mathcal H}}
\def\ck{{\mathcal K}}

\def\cN{{\mathcal N}}
\def\cR{{\mathcal R}}
\def\cU{{\mathcal U}}
\def\cV{{\mathcal V}}
\def\cw{{\mathcal W}}

\def\cv{{\mathcal V}}

\def\R{\mathbb{R}}
\def\A{\mathbb{A}}
\def\Z{\mathbb{Z}}
\def\N{\mathbb{N}}

\def\T{\mathbb{T}}
\def\Q{\mathbb{Q}}

\def\smallskip{\par\vspace{1mm}}
\def\medskip{\par\vspace{2mm}}
\def\bigskip{\par\vspace{3mm}}

\newcount\equationcount
\def\thenumber{0}
\def\eq#1{\global\advance\equationcount by 1
   \def\thenumber{\number\equationcount}
                        {$$#1\eqno(\thenumber)$$}}

\tikzset{
xmin/.store in=\xmin, xmin/.default=-1.5, xmin=-1.5,
xmax/.store in=\xmax, xmax/.default=7.5, xmax=7.55,
ymin/.store in=\ymin, ymin/.default=-0.75, ymin=-0.75,
ymax/.store in=\ymax, ymax/.default=3.25, ymax=3.25,
}

\begin{document}

\title[Foliations by curves]{ {Actions of symplectic homeomorphisms/diffeomorphisms on foliations by curves in dimension 2}}

\author{Marie-Claude Arnaud$^{\dag,\ddag}$, Maxime Zavidovique$^{*,**}$}

\email{Marie-Claude.Arnaud@imj-prg.fr \\ maxime.zavidovique@upmc.fr}

\date{}



\thanks{$\dag$ Universit\'e Paris 7 -- Denis Diderot
Institut de Math\'ematiques de Jussieu-- Paris Rive Gauche
UMR7586
B\^atiment Sophie Germain
Case 7012
75205 PARIS Cedex 13
, FRANCE } 
\thanks{$\ddag$ member of the {\sl Institut universitaire de France.}}
\thanks{ $*$ IMJ-PRG, UPMC
4 place Jussieu,
Case 247
75252 Paris Cedex 5}
\thanks{ $**$ financ\' e par une bourse PEPS du CNRS}

\begin{abstract}  {The two main results of this paper concern the regularity of the invariant foliation of a $C^0$-integrable symplectic twist diffeomorphisms of the 2-dimensional annulus, namely that
\begin{itemize}
\item the generating function of such a foliation is $C^1$;
\item the foliation is H\" older with exponent $\frac{1}{2}$.
\end{itemize}
We also characterize foliations by graphs that are straightenable via a symplectic homeomorphism and prove that every symplectic homeomorphsim that leaves invariant all the leaves of a straightenable foliation has Arnol'd-Liouville coordinates, in which the Dynamics restricted to the leaves is conjugated to a rotation. We deduce that every Lipschitz integrable symplectic twist diffeomorphisms of the 2-dimensional annulus has Arnol'd-Liouville coordinates and  then provide examples of `strange' Lipschitz foliations in smooth curves that cannot be straightened by a symplectic homeomorphism and cannot be invariant by a symplectic twist diffeomorphism.

}
\end{abstract}

\maketitle
\section{Introduction and Main Results.}\label{SecIntro}
\subsection{Main results}
This article deals with foliations by curves in a 2-dimensional symplectic setting. The questions we raise for such a foliation are
\begin{itemize}
\item When is it (locally or globally) symplectically homeomorphic to the straight foliation\footnote{This will be precisely defined later.}?
\item What can be said on the foliation when it is invariant\footnote{In this article, we will say that a foliation is invariant by $f$ if every leaf is (globally) invariant.}  by a symplectic twist?
\item What can be said on a symplectic Dynamics that preserves such a foliation?
\end{itemize}
Before going more into details, let us explain our motivations.

The first problem in which we were interested is the possible extension of Arnol'd-Liouville theorem (see e.g. \cite{Duis}). This classical theorem concerns Hamiltonian Dynamics associated to a $C^2$ Hamiltonian function endowed with a complete system of independent commuting $C^2$ integrals. Then there exists an invariant $C^2$ foliation into Lagrangian submanifolds and in the neighbourhhood of every compact leaf of this foliation, there exist symplectic $C^1$ angle-action coordinates $H: (q, p)\in \cU\subset M\mapsto (\theta, I)\in U\subset \T^n\times \R^n$ such that in these coordinates
\begin{itemize}
\item the invariant foliation is the straight foliation $I=\text{constant}$;
\item the flow is $(\theta, I)\mapsto (\theta+t\triangledown h(I),I)$ where $h$ is a $C^2$ function.
\end{itemize}
In fact, there are two steps in this result.
\begin{itemize}
\item The first step consists in symplectically straightening the foliation via the chart $H$. The diffeomorphism $H$ is defined via its generating function $S(q, I)$.  We recall
$$H(q, p)=(\theta, I)\Longleftrightarrow \theta=\frac{\partial S}{\partial I}(q, I)\quad\text{et}\quad p=\frac{\partial S}{\partial q}(q, I).$$
A priori this generating function $S$ is only $C^2$ as the foliation was and the diffeomorphism $H$ is only $C^1$,  but because the invariant foliation is $C^2$, we can say a little more: when $I$ is fixed, $\Phi=H^{-1}$ is $C^2$ in the $\theta$-variables
\item Then the second step consists in noticing that a symplectic flow that preserves every leaf of the straight foliation has to be a flow of rotations on every leaf.
\end{itemize}
In \cite{ArnaXue}, the hypothesis concerning the regularity of the invariant foliation was relaxed and the invariant foliation was just assumed to be $C^1$.  In this case, when the Hamiltonian satisfies the so-called A-non degeneracy condition (that contains the case of Tonelli Hamiltonians), the authors proved the existence of a symplectic homeomorphism $H$ straightening the invariant Lagrangian foliation, such that $H^{-1}$ is $C^1$ in the $\theta$ variable and such that the flow is written in the chart: $(\theta, I)\mapsto (\theta+t\triangledown h(I),I)$ where $h$ is a $C^1$ function.

Here we raise the problem of invariant $C^0$ foliation into invariant $C^0$-Lagrangian tori. In high dimension, the first problem is to define what is a $C^0$-foliation into Lagrangian tori. An interesting discussion on this topic is provided in the appendix of  \cite{ArnaXue}, but here we will consider the simplest case: in dimension 2, any foliation into curves can be seen as Lagrangian. Also we will assume that the foliations that we consider are not too complicated, because they are (at least locally in $C^1$ charts) foliations into graphs.

Even in this setting and for a symplectic twist diffeomorphism of the 2-dimensional annulus $\A=\T\times\R$, there exist results in which the authors are able to prove the existence of such an invariant continuous foliation into curves that are graphs, see e.g. \cite{CheSu} or \cite{FloLeC}, but not able to say more (e.g. to describe the Dynamics or prove that the foliation is symplectically straightenable).   Observe too that the case of a Tonelli Hamiltonian with two degrees of freedom corresponds to the case of twist maps by using a Poincar\'e section close to some invariant torus in an energy surface, and that in this setting also the same questions are open and relevant (see e.g. \cite{MaSor}).

\begin{defi}
A map $f:\A\rightarrow\A$ is {\em $C^0$-integrable} if $f$ has an invariant $C^0$-foliation into graphs.
\end{defi}

For such a foliation into graphs of $\theta\in\T\mapsto \eta_c(\theta)\in\R$ where $\int_\T\eta_c(\theta)=c$, we introduce what we  call its  generating function\footnote{This terminology will be better understood when we will introduce the generating functions of a large class of symplectic homeomorphisms. Let us mention that    for a general foliation, the generating function of the foliation is not necessarily the generating function of a symplectic homeomorphism. We will later give conditions for this to be true.} that is $u:\A\rightarrow \R$ defined by
$$u(\theta, c)=\int_0^\theta (\eta_c(t)-ct)dt.$$
\begin{defi}
Let $\cF$ be a continuous foliation of $\A$ into graphs. Then the  unique continuous function $u:\A\rightarrow \R$ that is $C^1$ with respect to the $\T$ variable such that 
\begin{itemize}
\item $\forall c\in\R, u(0, c)=0$,
\item $\forall c\in\R$, the graph of $c+\frac{\partial u}{\partial \theta}(., c)$ is a leaf of $\cF$,
\end{itemize}
is called   the {\em generating function} of $\cF$
\end{defi}

Our first result proves that in the $C^0$ integrable case of twist diffeomorphisms, there is more regularity of the generating function giving the foliation than we should expect.

\begin{theorem}\label{Tgeneder}
Let $f:\A\rightarrow\A$ be a  $C^1$ symplectic twist diffeomorphism. When $f$ is $C^0$ integrable, the generating function $u$ of its invariant foliation is  $C^1$  \\
Moreover, in this case,   we have\footnote{See the notation $\pi_1$ at the beginning of subsection \ref{ssnotatwist} .}
 \begin{itemize}
 \item the graph of  $c+\frac{\partial u}{\partial \theta}(., c)$ is a leaf of the invariant foliation;
 \item $h_c:\theta \mapsto \theta+\frac{\partial u}{\partial c}(\theta,c)$ is a semi-conjugacy between the projected Dynamics $g_c: \theta\mapsto \pi_1\circ f\big(\theta, c+\frac{\partial u}{\partial \theta}(\theta,c)\big)$ and a rotation $R$ of $\T$, i.e. $h_{c}\circ g_c=R\circ h_{c}.$
   \end{itemize}
 \end{theorem}
 This allows us to give an example of a foliation of the annulus into smooth graphs that cannot be invariant by a $C^0$-integrable symplectic twist diffeomorphism. But we will see in subsection \ref{sstrange} that it can be invariant by an exact symplectic twist homeomorphism that is a $C^1$-diffeomorphism.
 
 \begin{cor}\label{Corstrangefolia} Let $\varepsilon:\R\rightarrow \R$ be a non-$C^1$ function that is $\frac{1}{4\pi}$-Lipschitz. Then the function
$$(\theta,c)\mapsto u(\theta,c)=\frac{\varepsilon(c)}{2\pi}\sin(2\pi \theta)$$
is the generating function of  a  foliation of $\A$ into smooth graphs of $\theta\in\T\mapsto c+\varepsilon(c)\cos (2\pi \theta)$ that is invariant by no $C^0$-integrable symplectic twist diffeomorphism.
 \end{cor}

 The striking fact is the regularity in $c$. Indeed, if we have a $C^k$ foliation in graphs for some $k\geq 1$, we can only claim that $u$ and $\frac{\partial u}{\partial \theta}$ are $C^k$.  So in the $C^0$ case, even the derivability with respect to $c$ is surprising, which is a result of the invariance by a symplectic {\em twist} diffeomorphism. Also,  the fact that the semi-conjugacy $h_c$ continuously depends on $c$ even at a $c$ where the rotation number is rational is very surprising. At an irrational rotation number, this is an easy consequence of the uniqueness of the invariant measure supported on the corresponding leaf, but what happens for a rational rotation number is more subtle.
 
  Another result for $C^0$-integrable twist diffeomorphisms is that the invariant foliation is not only $C^0$, but also $\frac{1}{2}$-H\" older. It is well-known since Birkhoff that it is locally uniformly Lipschitz in the variable $\theta$ and we   prove here some regularity with respect to  $c$.
  \begin{theorem}\label{THolder}
Let $f:\A\rightarrow \A$ be a $C^1$ symplectic twist diffeomorphism that  is $C^0$ integrable with generating function $u$ of its invariant foliation.
Then on every compact subset of $\A$,  the foliation  $(\theta,c)\mapsto \eta_c(\theta)=c+\frac{\partial u}{\partial \theta}(\theta, c)$ is uniformly $\frac{1}{2}$-H\"older in the variable $c$. 
 \end{theorem}

 In the $C^0$-integrable case, the Dynamics restricted to a leaf with a rational rotation number is completely periodic. \\
 It is an open question if it can be a Denjoy counter-example  when restricted to a leaf with an irrational rotation number. \\
 With  the notations Theorem \ref{Tgeneder}, let us observe that
 when $f:\A\rightarrow \A$ is $C^0$ integrable, there exists a dense $G_\delta$ subset ${\cal G}$ of $\R$ such that for every $c\in {\cal G}$, the Dynamics restricted to the graph of $\eta_c$ is minimal. Indeed,  the set $\cR$ of recurrent points is a $G_\delta$ set with full Lebesgue measure, hence $\cR$ is  dense. Hence there exists a dense $G_\delta$ subset $G_1$ of $\R$ such that for every $c\in G_1$, the set $\big\{ \theta\in \T, \big(\theta, \eta_c(\theta)\big)\quad\text{is}\quad\text{recurrent}\big\}$ is a dense $G_\delta$ subset of $\T$. Hence, for $c\in G_1$, the Dynamics restricted to the graph of $\eta_c$ cannot be Denjoy. If we remove from $G_1$ the countable set of $c$'s that correspond to a rational rotation number, we obtain a dense $G_\delta$ subset of $\R$ such that  the Dynamics $f_{\vert \text{Graph}\textrm{$\big($}c+\frac{\partial u}{\partial \theta}(., c)\textrm{$\big)$}} $ is minimal.

 We will   give some conditions that imply that the Dynamics restricted to a leaf cannot be Denjoy. 
 
 \medskip
 Before this, we need to explain the notion of straightenable foliation. 
 
 \begin{notas}We will work in some open subsets $\cU$, $\cV$ of either $\A$ or $\R^2$, on which we have global symplectic coordinates that we denote by $(\theta, r)$ or $(\theta, c)$. Moreover, we will assume that $\cV=\{ (\theta, r); \theta\in(\alpha, \beta)\quad\text{and}\quad a(\theta)<r<b(\theta)\}$ or $\cV=\A$ where $a, b$ are some continuous functions and that $0\in(\alpha, \beta)$.
\end{notas}
 
Firstly, we  introduce the notion of exact symplectic homeomorphism, which is a particular case of the notion of symplectic homeomorphism that is due to Oh and M\" uller, \cite{OhMu}. Their notion coincides in this 2-dimensional setting with the one of orientation and Lebesgue measure preserving homeomorphism.  
 \begin{defi} An {\em exact symplectic homeomorphism}  from $\cU$ onto $\cV$  is a homeomorphism that is the  limit for the  for the $C^0$ compact-open topology of a sequence  of exact symplectic diffeomorphisms.\footnote{We recall that a diffeomorphism $f:\A\rightarrow \A$ is exact symplectic if   the 1-form $f^*(rd\theta)-rd\theta$ is exact.}.

\end{defi}  

\begin{remk}
\begin{itemize}
\item In $\R^2$, every 1-form is exact and then the notions of symplectic homeomorphisms and exact symplectic homeomorphisms coincide.
\item Let us recall that a symplectic diffeomorphism $f$ of $\A$ that is isotopic to identity is exact symplectic if and only if for every essential\footnote{An essential curve is a simple closed curve that is not homotopic to a point.}   curve $\gamma$ of the annulus, the algebraic area between $\gamma$ and $f(\gamma)$ is zero.
\end{itemize}
\end{remk}
A remarkable tool can be associated to the exact symplectic homeomorphisms that maps the standard horizontal foliation onto a foliation that is transverse to the vertical one.  This is called a generating function.

 \begin{theorem}[and definition]\label{TC0arn} We use the same meaning for $\cU$ and $\cV$ as in the notations.\\
 Let $\Phi:\cU\rightarrow \cV$ be an exact symplectic homeomorphism that maps the standard horizontal foliation onto a foliation $\cF$ that is transverse to the vertical one and that preserves the orientation  of the leaves\footnote{Here we means the orientation projected on the horizontal foliation.}.
 We use the notation $$\cw=\{ (\theta, c); \exists (x, c)\in \cU, \Phi(x, c)=(\theta, r)\in \cV\}=[\alpha, \beta]\times I$$ where $I$ is either  $ [c_-, c_+]$ or $ \R$.
  Then there exists a $C^1$ function $u:\cw\rightarrow \R$ such that $u(\beta,c)=u(\alpha,c)$\footnote{When $[\alpha, \beta]=\T$, we have $\alpha=\beta$ and there is no condition.} and
 $$\Phi(x, c)=(\theta, r)\Longleftrightarrow x=\theta+\frac{\partial u}{\partial c}(\theta, c)\quad{\rm and}\quad r=c+\frac{\partial u}{\partial \theta}(\theta, c).$$
In particular, when defined, every map $\theta\mapsto \theta +\frac{\partial u}{\partial c}(\theta, c)$ is injective and $\big\{\big(\theta, c+\frac{\partial u}{\partial \theta}(\theta, c)\big); (\theta, c)\in \cw\big\}$ is a leaf of the image of the standard horizontal foliation.\\
The function $u$ is called a {\em generating function} for $\Phi$.

Conversely,  we consider a  $C^0$-foliation of $\cV$ into graphs of $\eta_c$ for $c\in I$ and     assume that there exists a $C^1$ map $u : \cw \to \R$ such that 
\begin{itemize}
\item $u(0,c) = 0$ for all $c\in I$,
\item the graph of $\theta\mapsto c+ \frac{\partial u}{\partial \theta}(\theta,c)$ defines a leaf of the original foliation,
\item for all $c\in I$, the map $\theta \mapsto \theta + \frac{\partial u}{\partial c}(\theta , c)$ is injective.
\end{itemize}
Then the exact symplectic homeomorphism defined by
$$\Phi(x, c)=(\theta, r)\Longleftrightarrow x=\theta+\frac{\partial u}{\partial c}(\theta, c)\quad{\rm and}\quad r=c+\frac{\partial u}{\partial \theta}(\theta, c)$$
maps the standard horizontal foliation onto the original foliation. 
 
 \end{theorem}
 
 \begin{cor}\label{Lipnotstraighten} The foliation given in Corollary \ref{Corstrangefolia} cannot be straightened via an exact symplectic homeomorphism that preserves the horizontal orientation of the leaves.  Even locally, in any neighbourhood of points where $u$  is not $C^1$, it cannot be straightened via an exact symplectic homeomorphism that preserves the horizontal orientation of the leaves. 
 
 \end{cor} 
 
 \begin{remks}
 \begin{enumerate}
 \item The formulas of Theorem \ref{TC0arn} can be also written as follows.
 $$\Phi\Big(\theta + \frac{\partial u}{\partial c}(\theta , c), c\Big)=\Big(\theta, c+ \frac{\partial u}{\partial \theta}(\theta,c)\Big).$$

 \item Observe that Theorem \ref{Tgeneder} gives us a $C^1$ function $u$, but not the injectivity of $\theta \mapsto \theta + \frac{\partial u}{\partial c}(\theta , c)$. This is  why a priori the maps $h_c$ are not conjugacies, but only semi-conjugacies and in this case the restricted Dynamics may be Denjoy.
 \end{enumerate}
\end{remks}

We will now give a condition that implies that a foliation is straightenable by an exact symplectic homeomorphism.

\begin{defi}
 \begin{itemize}
\item A foliation into graphs $a\mapsto\eta_a$ is a   {\em Lipschitz foliation} if $(\theta, a)\mapsto \big(\theta, \eta_a(\theta)\big)$ is an homeomorphism that is locally biLipschitz; 
\item if $f$ has an invariant Lipschitz foliation, $f$  is {\em Lipschitz integrable}.
\end{itemize}
 \end{defi}
 The following proposition is a consequence of Theorem  \ref{TC0arn} and results of Minguzzi on the mixed derivative, \cite{Min2014}.
 
 \begin{propos}\label{PLip} 
Let $u:\A\rightarrow\R$ be the generating function of a continuous foliation of $\A$ into graphs. We assume that $u$   is   $C^1$. Then two following assertions are equivalent:
\begin{enumerate}
\item\label{pt1corLip} the foliation is  Lipschitz;
\item\label{pt2corLip}  we have
\begin{itemize}
\item  $\frac{\partial u}{\partial \theta}$ locally Lipschitz continuous; 
\item$\frac{\partial u}{\partial c}$ uniformly Lipschitz continuous in the variable  $\theta$ on any compact set of $c$'s;
\item  for every compact subset $\ck\subset \A$, there exists two constants $k_+>k_->-1$ such that $k_+\geq \frac{\partial^2 u}{\partial \theta\partial c}\geq k_-$  Lebesgue almost everywhere in $\ck$.
\end{itemize} 
\end{enumerate}
  In this case,  $u$ is the generating function of an exact symplectic homeomorphism   $\Phi:\A\rightarrow \A$  that  maps the standard foliation onto the invariant one. \end{propos}

Observe that Corollary \ref{Lipnotstraighten} gives an example of a Lipschitz foliation into smooth curves that is not straightenable via a symplectic homeomorphism. Hence the hypothesis that $u$ is $C^1$ is crucial in Proposition \ref{PLip}.

\begin{defi}
\begin{itemize}
\item A map $a\mapsto\eta_a$ defines a   {\em $C^k$ foliation} if $(\theta, a)\mapsto \big(\theta, \eta_a(\theta)\big)$ is a $C^k$ diffeomorphism;  if $f$ has an invariant $C^k$ foliation, $f$  is {\em $C^k$ integrable};
\item  following \cite{HiPuSh}, a map $a\mapsto\eta_a$ defines a   {\em $C^k$ lamination} if $(\theta, a)\mapsto \big(\theta, \eta_a(\theta)\big)$ is a homeomorphism,  every $\eta_a$ is $C^k$ and the map $a\mapsto \eta_a$ is continuous when $C^k(\T, \R)$ is endowed with the $C^k$ topology.
\end{itemize}
\end{defi}

\begin{cor}\label{bolle}
Let $k\geq 1$ and $r\mapsto f_r$ be a $C^k$-foliation in graphs. Then there exists a $C^{k-1}$ exact symplectic diffeomorphism\footnote{A $C^0$ diffeomorphism is a homeomorphism.} $\Phi : (\theta,r ) \mapsto \big(h(\theta,r) , \eta( h(\theta,r), r) \big)$ such that for each $r\in \R$, the set $\big\{ \big(\theta,\eta(\theta,r)\big), \theta \in \T\big\}$ is a leaf of the foliation. 
\end{cor}

 When a symplectic homeomorphism preserves a symplectic foliation that is symplectically straightenable, the Dynamics is very simple. Let us introduce a notion before explaining this point.
 
 \begin{defi}
  If $f:\A\rightarrow\A$ is a symplectic homeomorphism, {\em $C^0$ Arnol'd-Liouville coordinates} are given by a symplectic homeomorphism $\Phi:\A\rightarrow \A$ such that the standard foliation into graphs $\T\times \{ c\}$ is invariant by $\Phi^{-1}\circ f\circ \Phi$ and $$\Phi^{-1}\circ f\circ \Phi(x, c)=(x+\rho(c), c)$$ for some (continuous) function $\rho:\R\rightarrow \R$. 

 \end{defi}

\begin{propos}\label{Psymfolarn}
Let $f:\A\rightarrow \A$ be a symplectic homeomorphism that preserves (each leaf of) a 
$C^0$-foliation $\cF$. If the foliation is symplectically straightenable (by $\Phi : \A \to \A$ that maps the standard foliation  $\cF_0$ to $\cF=\Phi(\cF_0)$), then the homeomorphism $\Phi $ provides $C^0$ Arnol'd-Liouville coordinates.
 \end{propos}
 
 \begin{cor}\label{Corota}
A symplectic twist diffeomorphism $f: \A \to \A$ is $C^0$-integrable with the Dynamics on each leaf conjugated to a rotation if and only if the invariant foliation is exact symplectically homeomorphic to the standard foliation. In this case, $f$ admits global $C^0$ Arnol'd-Liouville coordinates. 
\end{cor}
 In the case where the invariant foliation by a symplectic twist diffeomorphism is Lipschitz, we are in the case of Proposition \ref{Psymfolarn} and so for every leaf, the restricted Dynamics is not  Denjoy.

 \begin{cor}\label{corLip} 
Let $f:\A\rightarrow\A$ be a symplectic twist diffeomorphism that is Lipschitz integrable. We denote by $u$ the generating function  of its invariant foliation.

Then  $u$ is the generating function of an exact symplectic homeomorphism   $\Phi:\A\rightarrow \A$  that maps the standard  foliation onto the invariant one such that:
$$\forall (x, c)\in \A,\quad  \Phi^{-1}\circ f\circ \Phi(x, c)=(x+\rho(c), c);$$
where $\rho:\R\rightarrow \R$ is an increasing biLipschitz homeomorphism.\\
Moreover, the invariant foliation is a $C^1$ lamination and $\Phi$ admit a partial derivative with respect to $\theta$.
The projected Dynamics $g_c$ restricted to every leaf is $C^1$ conjugated to a rotation via the $C^1$ diffeomorphism $h_c=Id_{\T}+\frac{\partial u}{\partial c}(\cdot, c):\T\rightarrow \T$ such that $h_{c}\circ g_c=R\circ h_{c}$. \end{cor}
Corollary \ref{corLip} provides some $C^0$ Arnol'd-Liouville coordinates.  A similar statement  for Tonelli Hamiltonians is proved in \cite{ArnaXue}, without the fact that the conjugation is $C^1$.

\subsection{Some notations and definitions}\label{ssnotatwist} We will use the following notations.

 \begin{notas}
 \begin{itemize}
 \item $\T=\R/\Z$ is the circle and $\A=\T\times \R$ is the annulus ;  $\pi: 
 \R\rightarrow \T$ is the usual projection;
 \item the universal covering of the annulus is denoted by $p:\R^2\rightarrow \A$;
 \item the corresponding projections are $\pi_1: (\theta, r)\in \A\mapsto \theta\in \T$ and $\pi_2: (\theta, r)\in \A\mapsto r\in \R$; we denote also the corresponding projections of the universal covering by $\pi_1$, $\pi_2~: \R^2\rightarrow \R$;
 \item the Liouville 1-form is defined on $\A$ as being $\lambda=\pi_2d\pi_1=rd\theta$; then $\A$ is endowed with the symplectic form $\omega=-d\lambda$.
 \end{itemize}
\end{notas}

  Let us give  the definition of an exact symplectic twist diffeomorphism.
\begin{defi}
An {\em  symplectic  twist diffeomorphism} $f:\A\rightarrow \A$ is a $C^1$ diffeomorphism such that
\begin{itemize}
\item $f$ is isotopic to identity;
\item $f$ is  symplectic, i.e. if $f^*\omega=\omega$;
\item $f$ has the {\em twist property} i.e. if $F=(F_1, F_2):\R^2\rightarrow \R^2$ is any lift of $f$, for any $\tilde \theta\in \R$, the map $r\in \R\mapsto F_1(\tilde \theta, r)\in \R$ is an increasing $C^1$ diffeomorphism from $\R$ onto $\R$.
\end{itemize}
\end{defi}

\subsection{Content of the different sections} The main tools that will use are tools of ergodic theory, symplectic (continuous or differentiable) Dynamics, in particular symplectic twist maps, Green bundles.
  Let us detail what will be in the different sections
\begin{itemize}
\item Section \ref{secfirst} contains the proof of of Theorem \ref{Tgeneder}; after recalling some generalities about symplectic twist diffeomorphisms, we consider the case of the rational curves by using some ergodic theory, then we prove  regularity of $u$ by using also ergodic theory;
\item Theorem  \ref{THolder}  is proved in section \ref{SHolder};
\item Theorem \ref{TC0arn} is proved in section \ref{SCOarna};
\item Section \ref{Sfollip} contains the proofs of   Proposition \ref{PLip} and Corollary \ref{bolle};
\item Section \ref{shomeomC0} contains   the proofs of 
Proposition \ref{Psymfolarn}, Corollary \ref{Corota}, Corollary  \ref{corLip};
\item Section \ref{sstrange} introduces a strange foliation, provides the proofs of Corollaries \ref{Corstrangefolia} and \ref{Lipnotstraighten} and an example of an 
 exact symplectic twist map that leaves the strange foliation invariant;
 \item Appendix \ref{AppA} contains an  example of a foliation by graphs that is the inverse image of the standard foliation by a symplectic map but not by a symplectic  homeomorphism and Appendix  \ref{ssGreenb} recalls some results about Green bundles.
\end{itemize}

\subsection*{Acknowledgements} The authors are grateful to Philippe Bolle for insightful discussions that helped clarify and simplify some proofs of this work.

\section{Proof of   Theorem \ref{Tgeneder}}\label{secfirst}

We assume that $f:\A\rightarrow \A$ is a $C^k$ symplectic twist diffeomorphism  (with $k\geq 1$) that has a continuous invariant foliation into continuous graphs with generating function denoted by $u$. We write $\eta_c=c+\frac{\partial u}{\partial \theta}(\cdot, c)$ and we recall that 
Birkhoff's theorem (see \cite {Bir1}, \cite{Fa1}  and \cite {He1}) implies that all the $\eta_c$ are Lipschitz.\\
\begin{nota}
For every $c\in\R$, we will denote by $g_c:\T\rightarrow \T$ the restricted-projected Dynamics to the graph of $\eta_c$, i.e
$$g_c(\theta)=\pi_1\circ f\big(\theta, \eta_c(\theta)\big).$$
\end{nota}
\subsection{Some generalities}\label{ss31}\hglue 1 truemm

\begin{notas}
\begin{itemize}
\item In $\R^2$ we denote by $B(x, r)$ the open disc for the usual Euclidean distance with center $x$ and radius $r$;
\item we denote by $R_\alpha:\T\rightarrow \T$ the rotation $R_\alpha(\theta)=\theta+\alpha$;
\item if E is a finite set, $\sharp(E)$ is the number of elements it contains;
\item we denote by $\lfloor \cdot\rfloor : \R\rightarrow \Z$ the integer part.
\end{itemize}


We fix a lift $F:\R^2\rightarrow \R^2$ of $f$.   We denote by $\tilde \eta_c:\R\rightarrow \R$ the lift of $\eta_c$.
We denote by $\rho$ the function that maps $c\in \R$ to the rotation number $\rho(c)\in\R$ of the restriction of $F$ to the graph of   $\tilde\eta_c$.
\end{notas}
 The map $\rho$  is then an increasing homeomorphism. \\
When moreover the foliation is biLipschitz, we will prove that $\rho$ is a biLipschitz homeomorphism (see Proposition \ref{rho}).  
  We recall a well-known result concerning the link between invariant measures and semi-conjugacies for orientation preserving homeomorphisms of $\T$.

\begin{proposition}\label{Pconjmeas} Assume that $\mu_c$ is a non-atomic Borel invariant probability measure by $g_c$. Then,  if $\rho(c)$  is irrational or $g_c$ is $C^0$ conjugated to a rotation,
 the map $h_c:\T\rightarrow \T$ defined by $h_c(\theta)=\int_0^\theta d\mu_c$ is a semi-conjugacy between $g_c$ and the rotation with angle $\rho(c)$, i.e:
$$h_c\big(g_c(\theta)\big)=h_c(\theta)+\rho(c).$$

\end{proposition}

\begin{proof}
Let $\tilde \mu_c$ be the pull back  measure of $\mu_c$  to $\R$ and let $\tilde g_c:\R\rightarrow \R$ be a lift of $g_c$ to $\R$. Then we have for every $\Theta\in [0, 1]$ lift of $\theta\in \T$:
$$\tilde \mu_c([0, \Theta])=\tilde \mu_c([\tilde g_c(0), \tilde g_c(\Theta)])=\tilde \mu_c\big(\big[ \lfloor \tilde g_c(0)\rfloor, \tilde g_c(\Theta)\big]\big)-\tilde \mu_c\big(\big[\lfloor \tilde g_c(0)\rfloor, \tilde g_c(0)\big]\big);$$
where $\lfloor\tilde g_c(0)\rfloor$ is the integer part of $\tilde g_c(0)$. This implies\footnote{ { Recall that if $f : \T \to \T$ is an orientation preserving homeomorphism then either $\rho(f)$ is irrational, $f$ is semi-conjugated (by $h$) to the rotation $R_{\rho(f)}$  and the only invariant measure is the pull back of the Lebesgue measure by $h$; or $\rho(f)$ is rational and the invariant measures are supported on periodic orbits.  When $\rho(f)$ is irrational or when $f$ is $C^0$ conjugate to a rotation and $\rho(f)\in [0, 1[$, then for any invariant measure $\mu$ and $x\in \T$, $\mu([x,f(x)[) = \rho(f)$. }  }
$$h_c(\theta)=h_c\big(g_c(\theta)\big)-\tilde\mu_c\big(\big[0, g_c(0)\big]\big)=h_c\big(g_c(\theta)\big)-\rho(c).$$
Moreover, as we assumed that $\mu_c$ is non-atomic, $h_c$ is continuous.
\end{proof}

\begin{remks}
\begin{enumerate}
\item In the other sense, if $h_c$ is a  (non-decreasing) semi-conjugacy such that $h_c\circ g_c=h_c+\rho(c)$, then $\mu([0, \theta])=h_c(\theta)-h_c(0)$ defines a $g_c$-invariant Borel probability measure;
\item When $\rho(c)$ is irrational, it is well known that the Borel invariant probability measure $\mu_c$ is unique and that the semi-conjugacy $h_c$ is unique up to constant.  
\end{enumerate}
\end{remks}
\begin{nota}
When $\rho(c)$ is irrational, we will denote by $h_c$ the semi-conjugacy such that $h_c(0)=0$.
\end{nota}


Before entering the core of the proof, let us mention a useful fact about iterates of $C^0$-integrable symplectic twist diffeomorphisms:
\begin{proposition}\label{iterateTwist}
Let $f: \A \to \A$ be a $C^0$-integrable $C^1$  symplectic twist diffeomorphism, then so is $f^n$ for all $n>0$.
\end{proposition}
 This is specific to the integrable case: { in general, an iterated  twist diffeomorphism  is not a  twist diffeomorphism  as can be seen in the neighborhood of an elliptic fixed point.}
\begin{proof}
We argue by induction on $n>0$. The initialization being trivial, let us assume the result true for some $k>0$. Let $F : \R^2 \to \R^2$ be a lift of $f$. For any $c\in \R$ using the notations given at the beginning of section \ref{secfirst}, we have 
$$\forall \theta \in \T, \forall m>0 , \quad f^m\big(\theta, \eta_c(\theta)\big) = \big(g_c^m(\theta), \eta_c\circ g_c^m(\theta)\big).$$
 Observe that if $f^m$ satisfies the twist condition and $c_1<c_2$ are two real numbers, then we have
$$\tilde g^m_{c_1}(t)=\pi_1\circ F^m\big(t, \eta_{c_1}(t)\big)<\pi_1\circ F^m\big(t, \eta_{c_2}(t)\big)=\tilde g^m_{c_2}(t)$$
and $\displaystyle{\lim_{t\rightarrow\pm\infty}\tilde g_{c_1}(t)=\pm\infty}$.

 Let $c_1<c_2$ and $t\in \R$. Denoting with $\sim$ the lifts of the considered functions we obtain that 
$$ \pi_1\big(F^{n+1}(t,c_2)\big) -\pi_1\big(F^{n+1}(t,c_1)\big) = \tilde g_{c_2}\circ \tilde g^{n}_{c_2}(t) -\tilde g_{c_1}\circ  \tilde g^{n}_{c_1}(t) \geq  \tilde g_{c_2}\circ \tilde g^{n}_{c_1}(t) -\tilde g_{c_1}\circ \tilde g^{n}_{c_1}(t) ,$$
where we have used the induction hypothesis and the fact that $\tilde g_{c_2}$ is increasing. It follows that $c\mapsto  \pi_1\big(F^{n+1}(t,c)\big) $ is an increasing diffeomorphism on its image.  Observe also that this inequality implies that $\displaystyle{\lim_{c_2\rightarrow+\infty} \pi_1\big(F^{n+1}(t,c_2)\big) =+\infty}$ because $\displaystyle{\lim_{c_2\rightarrow +\infty}\tilde g_{c_2}(s) =+\infty}$. Moreover 
$$ \pi_1\big(F^{n+1}(t,c_2)\big) -\pi_1\big(F^{n+1}(t,c_1)\big) = \tilde g_{c_2}\circ \tilde g^{n}_{c_2}(t) -\tilde g_{c_1}\circ  \tilde g^{n}_{c_1}(t) \leq  \tilde g_{c_2}\circ \tilde g^{n}_{c_2}(t) -\tilde g_{c_1}\circ \tilde g^{n}_{c_2}(t) ,$$
implies that $\displaystyle{\lim_{c_1\rightarrow-\infty} \pi_1\big(F^{n+1}(t,c_1)\big) =-\infty}$ because $\displaystyle{\lim_{c_1\rightarrow -\infty}\tilde g_{c_1}(s) =-\infty}$. So finally $c\mapsto  \pi_1\big(F^{n+1}(t,c)\big) $ is an increasing diffeomorphism onto $\R$.

\end{proof}

\subsection{Differentiability and conjugacy along the rational curves}\label{ssrat}
It is proved in \cite{Arna1} that for every $r=\frac{p}{q}\in\Q$, $\eta_c=\eta_{\rho^{-1}(r)}$ is   $C^k$   and the restriction of $f$ to the graph $\Gamma_c$ of $\eta_c$ is completely periodic: $f^q_{|\Gamma_c}={\rm Id}_{\Gamma_c}$. Moreover, along these particular curves, the two Green bundles (see Appendix  \ref{ssGreenb} for definition and results) are equal:
$$G_-\big(\theta, \eta_c(\theta)\big)=G_+\big(\theta, \eta_c(\theta)\big).$$
  \begin{theorem}\label{Tdiffrat}
  \begin{itemize}
\item Along every leaf $\Gamma_c$ such that $\rho(c)\in\Q$, the derivative $\frac{\partial \eta_c(\theta)}{\partial c}=1+\frac{\partial^2u}{\partial c\partial \theta}>0$ exists and  $C^{k-1}$  depends on $\theta$;
\item for any $c$ such that $\rho(c)$ is rational, the measure $\mu_c$ on $\T$ with density $\frac{\partial \eta_c}{\partial c}$ is a Borel probability measure invariant by $g_c$ and  for $\theta\in [0, 1]$, the equality
$$h_c(\theta)=\mu_c([0, \theta])=\int_0^\theta  \frac{\partial \eta_c}{\partial c}(t)dt$$ 
defines a conjugacy between $g_c$ and the rotation with angle $\rho(c)$;
\item then the map $c\in\R\mapsto \mu_c$ is continuous and also $c\in\R\mapsto h_c$ for the uniform $C^0$ topology. Thus $(\theta, c)\mapsto h_c(\theta)$ is continuous.
\end{itemize}
\end{theorem}

\begin{remks}
\begin{enumerate}
\item Observe that because $c\mapsto \eta_c$ is increasing, we know that for Lebesgue almost every $(\theta, c)\in \A$, the derivative $\frac{\partial \eta_c(\theta)}{\partial c}$ exists (see \cite{KFT}). But our theorem says something different.
\item Because of the continuous dependence on $\theta$  along the rational curve,  we obtain that $\frac{\partial \eta_c(\theta)}{\partial c}$ restricted to every rational curve is bounded (that is clear when we assume that the foliation is Lipschitz but not if the foliation is just continuous).
\end{enumerate}
\end{remks}

\noindent{\em Proof of the first point.}
  We fix $A\in \R$ such that $\rho(A)=\frac{p}{q}\in \Q$. Replacing $f$ by $f^q$, we can assume that $\rho(A)\in\Z$. Observe that because of the $C^0$-integrability of $f$, $f^q$ is also a ($C^0$-integrable with the same invariant foliation) $C^k$  symplectic twist diffeomorphism  (Proposition \ref{iterateTwist}).  

We define $G_A: \A\rightarrow \A$ by 
\begin{equation}\label{Ecurvezero} G_A(\theta, r)=\big(\theta, r+\eta_A(\theta)\big).\end{equation}
Then $G_A^{-1}\circ f^q\circ G_A$ is also a $C^0$-integrable $C^k$  symplectic twist diffeomorphism   and $\T\times\{ 0\}$ is filled with fixed points.

We finally have to prove our theorem in this case and we use the notation $f$ instead of    $G_A^{-1}\circ f^q\circ G_A$. We can assume that $A=0$  instead of $A\in\Z$.

Because of the semi-continuity of the two Green bundles $G_-=\R(1, s_-)$ and $G_+=\R(1, s_+)$, we have that for any point $x=(\theta, r)$ sufficiently close to $\T\times \{ 0\}$: $\max\{ |s_-(x)|, |s_+(x)|\} <\varepsilon$ is small.

Now we fix $c$ small and consider for every $\theta\in\T$ the small triangular domain $\ct(\theta)$ that is delimited by the three following  curves

\begin{itemize}
\item the graph of $\eta_c$;
\item the vertical $\cv_\theta=\{\theta\}\times \R$;
\item the image $f(\cv_\theta)$ of the vertical at $\theta$.
\end{itemize}
\begin{center}
\includegraphics[width=5cm]{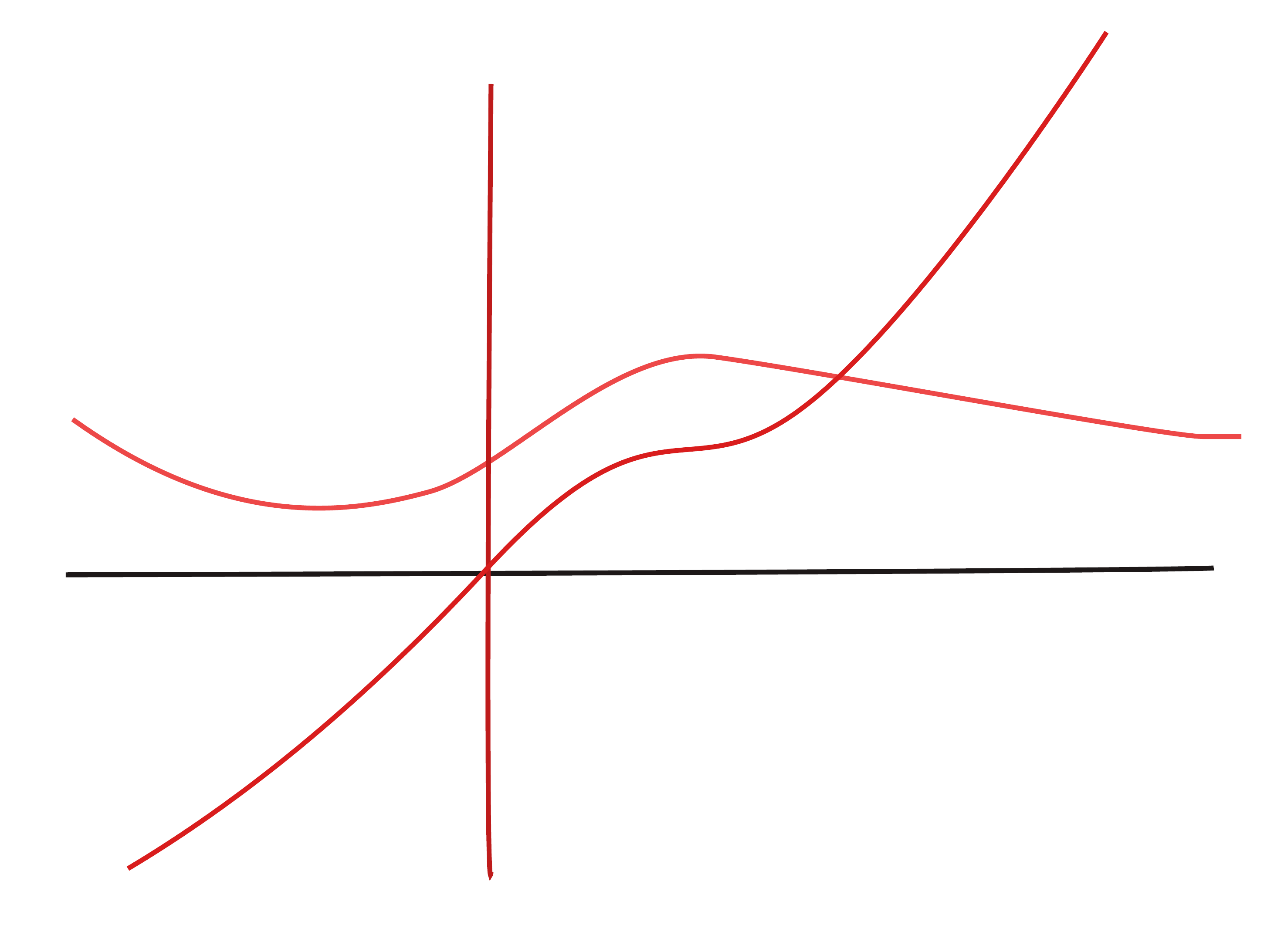}
\end{center}
 To be more precise, $\ct(\theta)$ is `semi-open' in the following sense; it contains  its whole boundary except the image $f(\cv_\theta)$ of the vertical at $\theta$.\\
We assume that $c>0$. The case $c<0$ is similar.\\
As the slope of $\eta_c$ is almost $0$ (because between the slope of the two Green bundles,  see Proposition \ref{PGreensand}) and the slope of the side of the triangle that is in $f(\cv_\theta)$ is almost $\frac{1}{s(\theta)}$ where $s(\theta)>0$ is the torsion that is defined by
\begin{equation}\label{Ematf} Df(\theta, 0)=\begin{pmatrix} 1&s(\theta)\\
0& 1\end{pmatrix},\end{equation}
the area of this triangle is 
\begin{equation}\label{E1}\lambda\big(\ct(\theta)\big)=\frac{1}{2}\big( \eta_c(\theta)\big)^2\big( s(\theta)+\varepsilon (\theta, c)\big);\end{equation}
where 
$${\rm uniformly}\quad{\rm for}\quad \theta \in \T,\quad  \lim_{c\rightarrow 0}\varepsilon (\theta, c)=0.$$
Let $\lambda$ be the Lebesgue measure restricted to the  invariant sub-annulus 
$$\ca_c=\bigcup_{\theta\in\T} \{ \theta\}\times [0, \eta_c(\theta)].$$
Being symplectic,   $f$ preserves $\lambda$. Moreover, every ergodic measure $\mu$ for $f$ with support in $\ca_c$ is supported on one curve $\Gamma_A$ with $A\in [0, c]$. But $f_{|\Gamma_A}$ is semi-conjugated to a rotation with an  angle  $\rho(A)$ that is in $[0, 1[$. Hence every interval in $\Gamma_A$ that is between some $\big(\theta, \eta_A(\theta)\big)$ and $f\big(\theta, \eta_A(\theta)\big)$ has the same $\mu$-measure, which is just given by the rotation number $\rho(A)\in [0, 1]$  on the graph of $\eta_A$. This implies that $\theta\mapsto \mu\big(\ct(\theta)\big)$ is constant. Hence for every $\theta, \theta'\in \T$ and for every ergodic measure $\mu$ with  support in $\ca_c$, we have $\mu\big(\ct(\theta)\big)=\mu\big(\ct(\theta')\big)$.\\
 We now use  the ergodic decomposition of invariant measures (see e.g. \cite{Man1}). Let $\delta_x$ be the notation for the Dirac measure at $x$. For every $a\in \A$, we denote   $\displaystyle{\lim_{n\rightarrow \infty}\frac{1}{n}\sum_{k=1}^{n-1}\delta_{f^n(a)}}$  by $\lambda_a$. Then we have  $\lambda=\int\lambda_ad\lambda(a)$. Denoting the rotation number of a point $a\in\A$ by  $\cR(a)$, we deduce that:
\begin{equation}\label{Airetriangle}\forall \theta, \theta'\in \T,\quad  \lambda\big(\ct(\theta)\big)=\lambda\big(\ct(\theta')\big)=\int_{\ca_c} \cR(a)d\lambda(a).\end{equation}

We deduce from equation (\ref{E1})  that 
$${\rm uniformly}\ \ {\rm for}\quad \theta, \theta'\in\T,\quad  \lim_{c\rightarrow 0} \frac{\eta_c(\theta')}{\eta_c(\theta)}=\sqrt{\frac{s(\theta)}{s(\theta')}}.$$
 Integrating with respect to $\theta'$, we deduce that uniformly in $\theta$, we have
$$ \lim_{c\rightarrow 0} \frac{c}{\eta_c(\theta)}=\sqrt{s(\theta)}\int_\T \frac{dt}{\sqrt{s(t)}}.$$

This implies that 
\begin{equation}\label{neq 0}
\frac{\partial \eta_c(\theta)}{\partial c}_{|c=0} =\Big( \int_\T \frac{dt}{\sqrt{s(t)}} \Big)^{-1}\frac{1}{\sqrt{s(\theta)}};
\end{equation}
and even
\begin{equation}\label{Eeta} \eta_c(\theta)=c\Big( \int_\T \frac{dt}{\sqrt{s(t)}} \Big)^{-1}\bigg(\frac{1}{\sqrt{s(\theta)}}+\varepsilon(\theta,c)\bigg)
\end{equation}
where
$${\rm uniformly}\ \ {\rm for}\quad \theta \in \T,\quad \lim_{c\rightarrow 0}\varepsilon (\theta, c)=0.$$

Observe that $ \frac{\partial \eta_c}{\partial c}=\Big( \int_\T \frac{dt}{\sqrt{s(t)}} \Big)^{-1}\frac{1}{\sqrt{s( \cdot)}}$ is a  $C^{k-1}$  function of $\theta'$. This proves the first point of theorem \ref{Tdiffrat}.

\noindent{\em Proof of the second point.} We deduce  from the first point that for any $c$ such that $\rho(c)$ is rational, the function  $\frac{\partial \eta_c}{\partial c} $ is continuous and positive. Moreover, its integral on $\T$ is 1. Hence $\frac{\partial \eta_c}{\partial c} $  is the density of a Borel probability measure that is equivalent to Lebesgue.  We now introduce:

\begin{nota}
If $c<c'$, we denote by $\Lambda_{c, c'}$ the normalized Lebesgue measure between the graph of $\eta_c$ and the graph of $\eta_{c'}$.
\end{nota}
Then $f$ preserves $\Lambda_{c, c'}$. Observe that   for any measurable $I\subset \T$, we have
\begin{equation}\label{E4}\Lambda_{c, c'} \big(\{ (\theta, r); \theta \in I, \  r\in [\eta_c(\theta), \eta_{c'}(\theta)]\}\big)=\frac{1}{c-c'}\int_I \big(\eta_c(\theta)-\eta_{c'}(\theta)\big)d\theta.
\end{equation}
\begin{lemma}\label{Lratmeas}
If $\rho(c)$ is rational, then $\displaystyle{\lim_{c'\rightarrow c}\Lambda_{c, c'}}$ is a measure supported on the graph of $\eta_c$ whose projected measure  $\mu_c$   has  density   $\frac{\partial \eta_c}{\partial c}$ with respect to the  Lebesgue  measure of $\T$.

Hence if $h_c(\theta)=\int_0^\theta  \frac{\partial \eta_c}{\partial c}(t)dt$, we have $$h_c\circ \pi_1\circ f\big(\theta, \eta_c(\theta)\big)=h_c(\theta)+\rho(c).$$
\end{lemma}

\begin{proof}
 Using  Equation (\ref{Eeta}),  we can take the limit in Equation (\ref{E4}) or more precisely for any $\psi\in C^0(\A, \R)$ in 
$$\int \psi(\theta, r)d\Lambda_{c, c'}(\theta, r)=\int_\T\frac{1}{c-c'}\bigg( \int_{\eta_{c'}(\theta)}^{\eta_c(\theta)}\psi(\theta, r)dr\bigg) d\theta$$
 and obtain that the limit is an invariant measure supported in the graph of $\eta_c$ whose projected measure $\mu_c$ has a density with respect to Lebesgue that is equal to
$\frac{\partial \eta_c}{\partial c}$. We then use Proposition \ref{Pconjmeas} to conclude that $h_c$ is the wanted conjugacy.
 \end{proof}
\noindent   {\em Proof of the third point.} 
We noticed that when $\rho(c)$ is irrational, there is only one invariant Borel probability measure that is supported on the graph of $\eta_c$. This implies the continuity of the map $c\mapsto \mu_c$ at such a $c$. Let us look at what happens when $\rho(c)$ is rational.
 
\begin{proposition}
For every $c_0\in \R$ such that $\rho(c_0)$ is rational, for every $\theta\in [0, 1]$, we have
$$\lim_{c\rightarrow c_0} \mu_c([0, \theta])=\mu_{c_0}([0, \theta])$$
and the limit is uniform in $\theta$.
\end{proposition}
This joint with the continuity of $h_{c_0}$ implies the continuity of $(\theta, c)\mapsto h_c(\theta)$ at $(\theta, c_0)$.\\
\begin{proof}  In this proof, we will use different functions $\varepsilon_i(\tau, c)$ and all these functions will satisfy uniformly in $\tau$
$$\lim_{c\rightarrow 0}\varepsilon_i(\tau, c)=0.$$
  As in the proof of the first point of Theorem \ref{Tdiffrat},  we can assume that $\eta_{c_0}=0$ (and then $c_0=0$)  and $\rho(0)=0$.\\
  We fix $\varepsilon>0$. Because of the continuity of  $\rho$, we can choose $\alpha$ such that if $|c|<\alpha$, then $|\rho(c)|<\varepsilon$.\\
 Let us introduce the notation $N_c=\lfloor \frac{1}{\rho(c)}\rfloor$ for $c\not=0$.    Let us assume that $c>0$ and $\theta\in (0, 1]$. 
We also denote by $  \tilde g_c$ the lift of $g_c$ such that $ \tilde g_c(0)\in[0, 1)$ and by $M_c(\theta)$
$$M_c(\theta)=\sharp \{ j\in \N;\quad  \tilde g _{c}^j(0)\in [0, \theta]\}.$$
 Hence, $M_c(\theta)$ is the number of points of the orbit of $0$ under $\tilde g_c$ that belong to $[0, \theta]$.  Observe that $M_c(\theta)$ is non-decreasing with respect to $\theta$.\\
As $\eta_c>0$, any primitive $\cN_c$ of $\eta_c$ is increasing, hence $M_c(\theta)$ is also the number of $\tilde g^k(0)$ such that $\cN_c\big(\tilde g^k(0)\big)$ belongs to $[\cN_c(0), \cN_c(\theta)]$, i.e.
\begin{equation}
\label{EestMC}
\begin{matrix} M_c(\theta)&=\sharp \Big\{ j\in  \N;\quad  \int_0^{\tilde g_{c}^j(0)}\eta_{c}(t)dt\leq  \int_0^{\theta}\eta_{c}(t)dt\Big\}\\
& =\sup\Big\{ j\in  \N;\quad  \int_0^{\tilde g_{c}^j(0)}\eta_{c}(t)dt\leq  \int_0^{\theta}\eta_{c}(t)dt\Big\}.\end{matrix}
\end{equation}
 Note that $M_c(1)=N_c$ because $g_c$ has rotation number $\rho(c)$ and that   we have $\forall \theta\in (0, 1]$, $M_c(\theta)\leq N_c$ as $M_c$ is non decreasing. We have also\\
$$\mu_c([0, \theta])=\sum_{j=0}^{M_c(\theta)-1}\mu_c([\tilde g_c^j(0), \tilde g_c^{j+1}(0)[)+\mu_c([\tilde g^{M_c(\theta)}(0), \theta])$$
and thus
$\mu_c([0, \theta])=M_c(\theta)\rho(c) +\Delta\rho(c)$ with $\Delta\in [0, 1]$  because $[\tilde g^{M_c(\theta)}(0), \theta]\subset [\tilde g^{M_c(\theta)}(0), \tilde g^{M_c(\theta)+1}(0)[$.\\
Hence
 \begin{equation}\label{Emesure}\mu_c([0, \theta])\in [ M_c(\theta)\rho(c), M_c(\theta)\rho(c)+\rho(c)]\subset \left[ \frac{M_c(\theta)}{N_c+1}, \frac{M_c(\theta)+1}{N_c}\right] .\end{equation}
Hence to estimate the measure $\mu_c([0, \theta])$ we need a good estimate of the number of $j$  such that  $\tilde g_{c}^j(0)$ belongs to $[0, \theta]$.
We have proved in Equation (\ref{Eeta}) that 
\begin{equation}
\label{Eetabis} \eta_{c}(\tau)=\Big(\int_\T \frac{dt}{\sqrt{s(t)}}\Big)^{-1}\frac{c\big(1+\varepsilon_0(\tau,c)\big)}{\sqrt{s(\tau)}}.
\end{equation}

 We deduce from Equation (\ref{Ematf}) that  $\tilde g_{c}(\tau)= \tau+\big(s(\tau)+\varepsilon_1(\tau, c)\big)\eta_c(\tau)$ where uniformly in $\tau$, we have: $\displaystyle{\lim_{c\rightarrow 0}\varepsilon_1(\tau, c)=0}$ and then by Equation (\ref{Eetabis}):
\begin{equation}\label{Epetitair}\int_\tau^{\tilde g_{c}(\tau)}\eta_{c}(t)dt=\eta_c(\tau)^2\big(s(\tau)+\varepsilon_2(\tau, c)\big) =\frac{c^2\big(1+\varepsilon_3(\tau, c)\big)}{\Big(\int_\T \frac{dt}{\sqrt{s(t)}}\Big)^2} .
\end{equation}
This says that the area between $\tau$ and $\tilde g_c(\tau)$ that is limited by the zero section and the graph of $\eta_{c}$ is almost constant (i.e. doesn't depend a lot on $\tau$).\\

 We deduce from Equation (\ref{EestMC})  that 
$$\int_0^{\tilde g_c^{M_c(\theta)}(0)}\eta_c(t)dt\leq\int_0^\theta\eta_c(t)dt<\int_0^{\tilde g_c^{M_c(\theta)+1}(0)}\eta_c(t)dt.$$
Hence 
$$\sum_{j=0}^{M_c(\theta)-1}\int_{\tilde g^j(0)}^{\tilde g^{j+1}(0)}\eta_c(t)dt\leq \int_0^\theta\eta_c(t)dt\leq 
\sum_{j=0}^{M_c(\theta)}\int_{\tilde g^j(0)}^{\tilde g^{j+1}(0)}\eta_c(t)dt.$$
Using Equation (\ref{Epetitair}), we deduce that
$$M_c(\theta)\frac{ c^2\big(1+\varepsilon_4(\theta, c)\big)}{\Big(\int_\T\frac{dt}{\sqrt{s(t)}}\Big)^2}\leq\frac{c\big(1+\varepsilon_5(\theta, c)\big)}{\int_\T\frac{dt}{\sqrt{s(t)}}}\int_0^\theta \frac{dt}{\sqrt{ s(t)}} <(M_c(\theta)+1)\frac{ c^2\big(1+\varepsilon_6(\theta, c)\big)}{\Big(\int_\T\frac{dt}{\sqrt{s(t)}}\Big)^2},
$$
and then 
\begin{equation}\label{EMC}M_c (\theta)
=\left\lfloor \frac{1}{c}\bigg(\Big( {\int_\T \frac{du}{\sqrt{s(u)}}}\Big)\Big(\int_0^{\theta}\frac{dt}{\sqrt{s(t)}}\Big)+\varepsilon_7(\theta, c)\bigg)\right\rfloor.\end{equation}
This implies that
\begin{equation}\label{ENC}N_c=M_c(1)=\left\lfloor \frac{1}{c}\bigg(\Big({\int_\T \frac{dt}{\sqrt{s(t)}}}\Big)^2+\varepsilon_8(1, c)\bigg)\right\rfloor\end{equation}
 and by Equations (\ref{neq 0}), (\ref{Emesure}), (\ref{EMC}) and (\ref{ENC}).
\begin{equation}\mu_c([0, \theta])=\frac{M_c(\theta)}{N_c}+\varepsilon_9(\theta, c)= \frac{\int_0^{\theta}\frac{dt}{\sqrt{s(t)}}}{\int_\T \frac{dt}{\sqrt{s(t)}}}+\varepsilon_{10}(\theta, c)=\mu_0([0,\theta])+\varepsilon_{11}(\theta, c).\end{equation}
As none of the measures $\mu_c$ has atoms, this implies that $c\mapsto \mu_c$ and all the maps $c\mapsto \mu_c([0, \theta])=h_c(\theta)$ are continuous. As every map $h_c$ is non decreasing in the variable $\theta$, we deduce from the Dini-Poly\`a Theorem  \cite[Exercise 13.b page 167]{Rudin} that $c\mapsto h_c$ is continuous for the $C^0$ uniform topology.
\end{proof}

\begin{remk}
If $\rho(A)=\frac{p}{q}$, then we proved that $\frac{\partial \eta_c(\theta)}{\partial c}_{\vert c=A} =  \Big(\int_\T \frac{dt}{\sqrt{\textrm{$s_q\big(t, \eta_A(t)\big)$}}}\Big)^{-1}\frac{1}{\sqrt{\textrm{$s_q\big(\theta, \eta_A(\theta)\big)$}}}$ where  $$Df^q(x)=\begin{pmatrix}
a_q(x)&s_q(x)\\
c_q(x)&d_q(x)
\end{pmatrix}.$$
Indeed, the term $s_q\big(t, \eta_c(t)\big)$ doesn't change when we conjugate by the map $G_A$ where $G_A(\theta, r)=\big(\theta, r+\eta_A(\theta)\big)$ as we did in section \ref{ssrat}.\\
This
 gives  for the conjugacy
$$ h_A(\theta)=\mu_A([0, \theta])=\bigg(\int_\T \frac{dt}{\sqrt{s_q\big(t, \eta_A(t)\big)}}\bigg)^{-1}\int_0^\theta \frac{1}{\sqrt{s_q\big(t, \eta_A(t)\big)}}dt.$$
Observe that this $C^k$ depends on $\theta$.\\
Observe too that  Equation (\ref{Eeta}) can be rewritten as
\begin{equation}\label{Eetab}\eta_c(\theta)=\eta_A(\theta)+(c-A)\bigg[\bigg(\int_\T \frac{dt}{\sqrt{s_q\big(t, \eta_A(t)\big)}}\bigg)^{-1}\frac{1}{\sqrt{s_q\big(\theta, \eta_A(\theta)\big)}}+\varepsilon(\theta,c)\bigg],
\end{equation}
where
$${\rm uniformly}\ \ {\rm for}\ \  \theta \in \T, \quad \lim_{c\rightarrow A}\varepsilon (\theta, c)=0.$$
Observe that the formula doesn't give  any continuous dependence of $h_c$ or $\frac{\partial \eta_c}{\partial c}$ in the $c$ variable, because $q$ can become very large when $c$ changes.

\end{remk}
\subsection{Generating function and regularity}  
To finish the proof of Theorem \ref{Tgeneder}, we have to prove that $u$ admits a derivative with respect to $c$ everywhere and that
$$\forall \theta\in\T, \forall c\in\R ,\quad h_c(\theta)=\theta+\frac{\partial u}{\partial c}(\theta, c).$$
Because we proved that $(\theta, c)\mapsto h_c(\theta)$ is continuous, we will deduce that $u$ is $C^1$.

 Observe that for every $\theta$, the map $c\mapsto u(\theta, c)+c\theta$ is increasing because every $  c\mapsto \eta_c(\theta)$ is increasing.
  \begin{theorem}
 The map $u$ is  $C^1$. 
 Moreover, in this case,  we have 
 \begin{itemize}
 \item the graph of $c+\frac{\partial u}{\partial \theta}(\cdot, c)$ is a leaf of the invariant foliation;
 \item $ \theta \mapsto \theta+\frac{\partial u}{\partial c}(\theta, c)$ is   the semi-conjugacy $h_c$ between $g_c$ and $R_{\rho(c)}$ given in Theorem \ref{Tdiffrat}. We have: $h_c\circ g_c=h_c+\rho(c)$.
 
  \end{itemize}
 \end{theorem}


 \begin{proof}  The first point is a consequence of the definition of $u$.\\
 Then $u(\cdot  ,c)$ and $\frac{\partial u}{\partial \theta}=\eta_{c}-c$ continuously depend on $(\theta, c)$.\\
 Observe that with the  notation (\ref{E4}), we have  
 \begin{multline*}
 \Lambda_{c, c'} \big(\{ (\theta, r); \ \ \theta \in [\theta_1, \theta_2],  r\in [\eta_{c}(\theta), \eta_{c'}(\theta)]\}\big)= \\
 =\frac{1}{c'-c}\Big( \big(u(\theta_2, c')-u(\theta_1, c')\big)-\big(u(\theta_2,c)-u(\theta_1,c)\big) \Big)+(\theta_2-\theta_1).
 \end{multline*}
  Moreover, if $\rho(c_0)\in \Q$, we deduce from   Lemma \ref{Lratmeas} that $u(\cdot , c)$   admits a derivative with respect to $c$ at $c_0$   $$\frac{\partial u}{\partial c}(\theta, c_0)=\lim_{c\rightarrow c_0}\frac{1}{c-c_0}\Big(\big(u(\theta,c)-u(0,c)\big)-\big(u(\theta,c_0)-u(0, c_0)\big)\Big)$$
 that is given by

   $$\frac{\partial u}{\partial c}(\theta, c_0)=\mu_{c_0}([0, \theta])-\theta=h_{c_0}(\theta)-\theta$$
  and this derivative continuously depends on $\theta$.\\
Assume now that $\rho(c_0)$ is   irrational and let $c$ tend to $c_0$. Every limit point of $\Lambda_{c,  c_0}$ when $c$ tends to $c_0$  is a Borel probability measure that is invariant by $f$ and supported on the graph of $\eta_{ c_0}$. As there exists only one such measure,  whose projection was denoted by $\mu_{c_0}$,  we deduce that 
$${ \pi_{1*}}\Big(\lim_{c\rightarrow c_0}\Lambda_{c,  c_0}\Big) =\mu_{c_0}.$$
 As $\mu_{c_0}$ has no atom, we have for all $\theta_0\in [0, 1)$
\begin{multline*}
  h_{c_0}(\theta_0)=\ \mu_{c_0}([0, \theta_0])  \\
\!\!\!\!\!\!\!\!\!\!\!\!\!\!=\lim_{c\rightarrow c_0}\Lambda_{c_0, c} (\{ (\theta, r); \theta \in [0, \theta_0],  r\in [\eta_{c_0}(\theta), \eta_{c}(\theta)]\}) \\
\qquad \quad =\lim_{c\rightarrow c_0}\frac{1}{c-c_0}\Big(\big(u(\theta_0,c)-u(0,c)\big)-\big(u(\theta_0, c_0)-u(0,c_0)\big)\Big)+\theta_0 \\
=\frac{\partial u}{\partial c}(\theta_0,c_0)+\theta_{0},
\end{multline*}
hence $u$ admits a derivative with respect to $c$  and 
$$h_{c_0}(\theta)=\mu_{c_0}([0, \theta])=\theta+\frac{\partial u}{\partial c}(\theta, c_0).$$

Because of Theorem \ref{Tdiffrat}, $(\theta, c)\mapsto \frac{\partial u}{\partial c}(\theta,c)=h_c(\theta)-\theta$ is continuous. As the two partial derivatives $\frac{\partial u}{\partial \theta}$ and $\frac{\partial u}{\partial c}$ are continuous in $(\theta, c)$, we conclude that $u$ is $C^1$.

\end{proof}

\section{Proof of Theorem \ref{THolder}}\label{SHolder}
We assume that  $f:\A\rightarrow \A$ is a $C^0$ integrable symplectic twist diffeomorphism   with generating function $u$ for its invariant foliation and use the notation $\eta_c(\theta)=c+\frac{\partial u}{\partial\theta}(\theta, c)$. We also denote the projected Dynamics on the graph of $\eta_c$ by $\tilde g_c(\theta)=\pi_1\circ F\big(\theta, \eta_c(\theta)\big)$ where we fix a lift $F:\R^2\rightarrow \R^2$ of $f$.
We work on some compact set
$$K=\big\{ \big(\theta, \eta_c(\theta)\big); \theta\in\T, c\in [c_1, c_2]\big\}.$$
Replacing $F$ by $F+(0, p)$ for some integer $p\in\N$, we can assume that the rotation number $\cR(x)$ of every $x\in K$ is positive.
We denote by $U=\{ (\theta, \eta_c(\theta)); \theta\in\T, c\in(c_1, c_2)\}.$
Being  $C^1$, $u$ is $C$-Lipschitz on $K$ for some constant $C>0$.
Using the notations of Appendix \ref{ssGreenb}, we recall that
\begin{equation}\label{EGreenusuel}\forall x\in \A, \forall n\in\N, \quad s_{-n}(x)<s_{-(n+1)}<s_-(x)\leq s_+(x)<s_{n+1}(x)<s_n(x)\end{equation}
and that all the maps $s_k$ are continuous. Hence there exists $b>a>0$ and $r\in (0, 1)$ such that
\begin{equation}\label{EGreenfin}
\forall x,y\in K, \quad d(x, y)<r\Rightarrow 0<a\leq s_1(x)-s_2(y)<s_1(x)-s_{-1}(y)\leq b.
\end{equation}

We deduce
$$\forall x,y\in K, 0<a\leq s_1(x)-s_+(y)<s_1(x)-s_{-}(y)\leq b.$$

Working in $\R^2$, we consider for any $c,c'\in (c_1, c_2)$ such that $c<c'$  the domain $D(\theta)$ whose boundary is the union of 
\begin{itemize}
\item a small piece $V$ of the vertical $\{\theta\}\times \R$ that is between $\eta_c$ and $\eta_{c'}$;
\item the arc $F(V)$;
\item pieces of $\eta_c$ and $\eta_{c'}$ that are between $V$ and $F(V)$.
\end{itemize}
\begin{center}
\includegraphics[width=10cm]{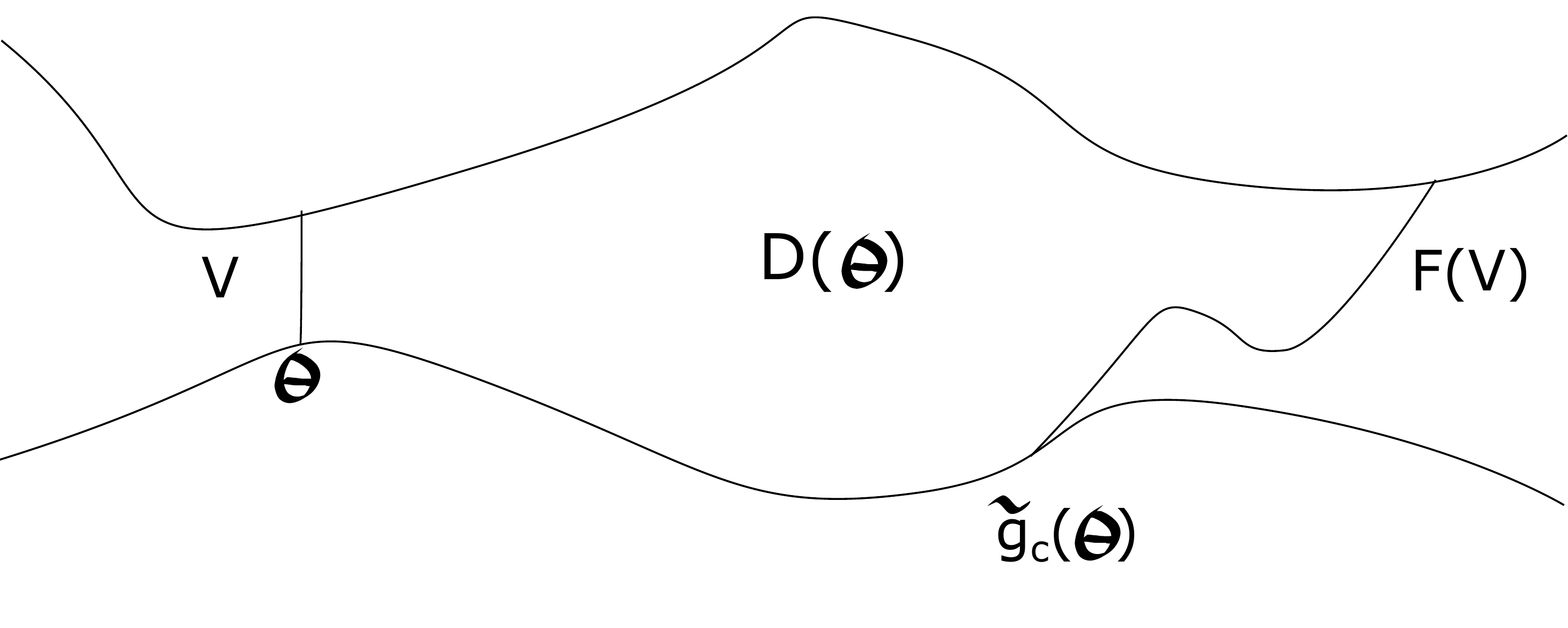}
\end{center}
Using the same method as in subsection \ref{ssrat} for the case of rational curve (that is the decomposition of Lebesgue measure into ergodic measures), we see that the area of $D(\theta)$ doesn't depend on $\theta$.\\
We then cut $D(\theta)$ into three subsets.
\begin{itemize}
\item if $p=\lfloor \tilde g_c(\theta)-\theta\rfloor$, $D_1(\theta)$ is the domain that is between $V$, $V+p$ and the graphs of $\eta_c$ and $\eta_{c'}$; observe that $p=\big\lfloor\cR \big(\theta, \eta_c(\theta)\big)\big\rfloor=\lfloor\rho(c)\rfloor$ doesn't depend on $\theta$\footnote{recall that when restricted to the graph of $\eta_c$, either $f$ is periodic and all points have the same period, or $f$ has no periodic orbit.};
\item $D_2(\theta)$ is the domain between $V+p$, the vertical $V^*$ at $F\big(\theta, \eta_c(\theta)\big)$  and the graphs of $\eta_c$ and $\eta_{c'}$;
\item $D_3(\theta)$ is the triangular domain between  $V^*$ , $F(V)$ and $\eta_{c'}$.
\end{itemize}
\begin{center}
\includegraphics[width=10cm]{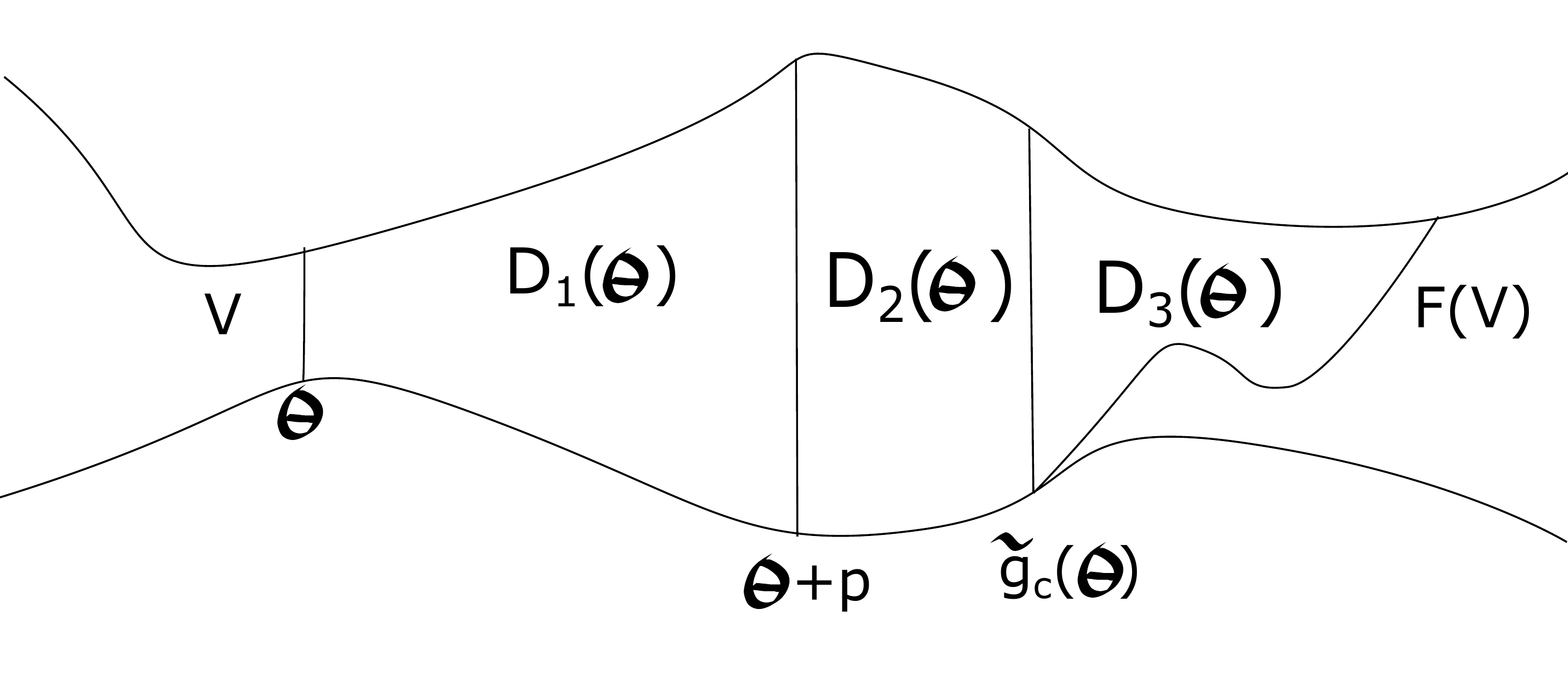}
\end{center}
Then the area of $D_1(\theta)$ is $p(c'-c)$ and doesn't depend on $\theta$. \\
The area of $D_2(\theta)$ is
$$\int_{\theta+p}^{ \tilde g_c(\theta)}\left( (c'+\frac{\partial u}{\partial \theta}(t, c')-(c+\frac{\partial u}{\partial \theta}(t, c)\right) dt,$$
a positive number equal to
$$(c'-c)( \tilde g_c(\theta)-\theta-p)+u( \tilde g_c(\theta), c')-u( \tilde g_c(\theta), c)+u(\theta+p, c)-u(\theta+p, c').$$
We recall that $u$ is $C$ Lipschitz on $K$ and that $ \tilde g_c(\theta)-\theta-p\in (0, 1)$. We deduce that the area of $D_2(\theta)$ belongs to 
$$\big(0, (2C+1)(c'-c)\big).$$
We want now to estimate the area of the triangle $D_3(\theta)$. This triangle has three (curved) sides that are
\begin{itemize}
\item the vertical side $V^*$ whose  length is equal to $\eta_{c'}\big( \tilde g_c(\theta)\big)-\eta_{c}\big( \tilde g_c(\theta)\big)$;
\item the side $F(V)$ with slope at $x\in F(V)$ equal to $s_1(x)$;
\item the side $H$ supported in $\eta_{c'}$ that is Lipschitz with tangent cone at $x$ that is contained in $[s_-(x), s_+(x)]\subset [s_{-1}(x), s_2(x)]$.
\end{itemize}
Because of the continuity of the foliation, there exists $\nu>0$ such that if $c_1\leq c\leq c'\leq c_2$ and $c'-c<\nu$, then the length of every piece of vertical $V$ between $\eta_c$ and $\eta_{c'}$ is less than $r$\footnote{Recall that $r$ was choosen to satisfy Formula    (\ref{EGreenfin}).}and the same is true for $F(V)$ because of the uniform continuity of $F$ on the strip between $\eta_{c_1}$ and $\eta_{c_2}$.  Using the fact that the tangent cone to the graph of $\eta_{c'}$ is between the two Green bundles (see Appendix \ref{ssGreenb}) and Equation (\ref{EGreenusuel}), we deduce that if $c'-c\in [0, \nu)$, $D_3(\theta)$ 
\begin{itemize}
\item is contained in a true triangle with vertical side equal to $V^*$, upper side with slope equal to $\displaystyle{\max_{x\in H}s_2(x)}$ and slope of lower side equal to $\displaystyle{\min_{x\in F(V)}s_1(x)};$
\item contains a true triangle with vertical side equal to $V^*$, upper side with slope equal to $\displaystyle{\min_{x\in H}s_{-1}(x)}$ and slope of lower side equal to $\displaystyle{\max_{x\in F(V)}s_1(x)};$

\end{itemize}

Observe that when the triangle is a true triangle, its horizontal height has length $\frac{\delta}{S-T}$ where $\delta$ is the length of the vertical side, $T$ is  the slope of the side coming from the upper point of the vertical side and $S$  is  the slope of the side coming from the lower point of the vertical side.  The area is then $\frac{\delta^2}{2(S-T)}$

These remarks and Equation (\ref{EGreenfin}) imply that the area of $D_3(\theta)$ belongs to the interval
$$\left[\frac{\Big(\eta_{c'} \big(\tilde g_c(\theta)\big)-\eta_{c}\big((\tilde g_c(\theta)\big)\Big)^2}{2b},\frac{\Big(\eta_{c'}\big((\tilde g_c(\theta)\big)-\eta_{c}\big((\tilde g_c(\theta)\big)\Big)^2}{2a}
\right].$$
Finally, the sum $A(\theta)$  of the area of $D_2(\theta)$ and $D_3(\theta)$ doesn't depend on $\theta$ and
\begin{itemize}
\item at a point (that always exists because $\int_\T(\eta_c-\eta_{c'})=c-c'$) such that $\eta_{c'}(\theta)-\eta_c(\theta)=c'-c$, we have 
$$A(\theta)\in\left[ \frac{(c'-c)^2}{2b},\frac{(c'-c)^2}{2a}+ (2C+1)(c'-c)
\right]$$ $$\subset \left[ \frac{(c'-c)^2}{2b},  \left( (2C+1)+\frac{1}{2a}\right)(c'-c)\right];
$$
\item at every point, we have $$A(\theta)\geq \textrm{area}\big(D_3(\theta)\big)\geq \frac{\Big(\eta_{c'}\big((\tilde g_c(\theta)\big)-\eta_{c}\big((\tilde g_c(\theta)\big)\Big)^2}{2b};$$
\end{itemize}
This implies that for $c_1\leq c\leq c'\leq c_2$ such that $c'-c<\nu$, we have
$$\forall \theta \in \R, \quad \frac{\big(\eta_{c'}(\theta)-\eta_{c}(\theta)\big)^2}{2b}\leq \left( (2C+1)+\frac{1}{2a}\right)(c'-c);$$
so
$$\forall \theta \in \R,\quad  \eta_{c'}(\theta)-\eta_{c}(\theta)\leq \sqrt{2b \left((2C+1)+\frac{1}{2a}\right)(c'-c)}.$$
Hence we obtain on the compact $K$ a uniform local constant of H\" older, and this implies that $\eta_c$ is uniformly $\frac{1}{2}$-H\"older in the variable $c$ on $K$.

\section{Proof of Theorem \ref{TC0arn}}\label{SCOarna}

Let us consider a $C^0$-foliation $\cF$ of $\cV=\{ (\theta, r); \theta\in(\alpha, \beta)\quad\text{and}\quad \eta_{c_-}(\theta)<r<\eta_{c_+}(\theta)\}$ or $\cV=\T$ into graphs: $(\theta, c)\in\cw \mapsto \big(\theta , \eta_c(\theta)\big)$, where $\frac{1}{\beta-\alpha}\int^\beta_\alpha \eta_c = c$ \big(when $\cV=\A$, $\cw=\A$ and in the other cases $\cw=(\alpha, \beta)\times (c_-, c_+)$\big). Then there exists a continuous function $u:\cw\rightarrow \R$ that admits a continuous derivative with respect to $\theta$ such that $\eta_c(\theta)=c+\frac{\partial u}{\partial \theta}(\theta,c)$ and $u(0, c)=0$, function that we called generating function of the foliation when $\cV=\A$. \\
\subsection{Proof of the first implication}
We assume that this foliation  is  exact symplectically homeomorphic to the standard horizontal foliation $\cF_0=\Phi^{-1}(\cF)$ by some exact symplectic homeomorphism $\Phi$.

Observe that the foliation $\cF$ is transverse to the ``vertical'' foliation $\cg_0$ into $\pi_1^{-1}(\{\theta\})$ for $\theta\in\pi_1(\cV)$. Hence the foliation $\cg=\Phi^{-1}(\cg_0)$ is a foliation of $\cU$ that is transverse to the standard (``horizontal'') foliation $\cF_0=\Phi^{-1}(\cF)$. This exactly means that the foliation  $\cg$ is a foliation into graphs of maps $\zeta_\theta: I=\pi_2(\cU)\rightarrow \pi_1(\cU)$. Hence there exists a continuous function $v:\cw\rightarrow \R$ that admits a continuous derivative with respect to $r$ such that the foliation $\cg$ is the foliation into graphs $\Phi^{-1}(\pi_2^{-1}(\{\theta\}))$ of $\zeta_\theta: r\mapsto \theta+\frac{\partial v}{\partial r}(\theta, r)$. Observe that by definition of $\zeta_\theta$, we have $\Phi\big(\zeta_\theta(c), c\big)=\big(\theta,\eta_c(\theta)\big)$. As a result, every map $\theta\mapsto \zeta_\theta(c)$ is a homeomorphism onto its image.
\begin{center}
\includegraphics[width=8cm]{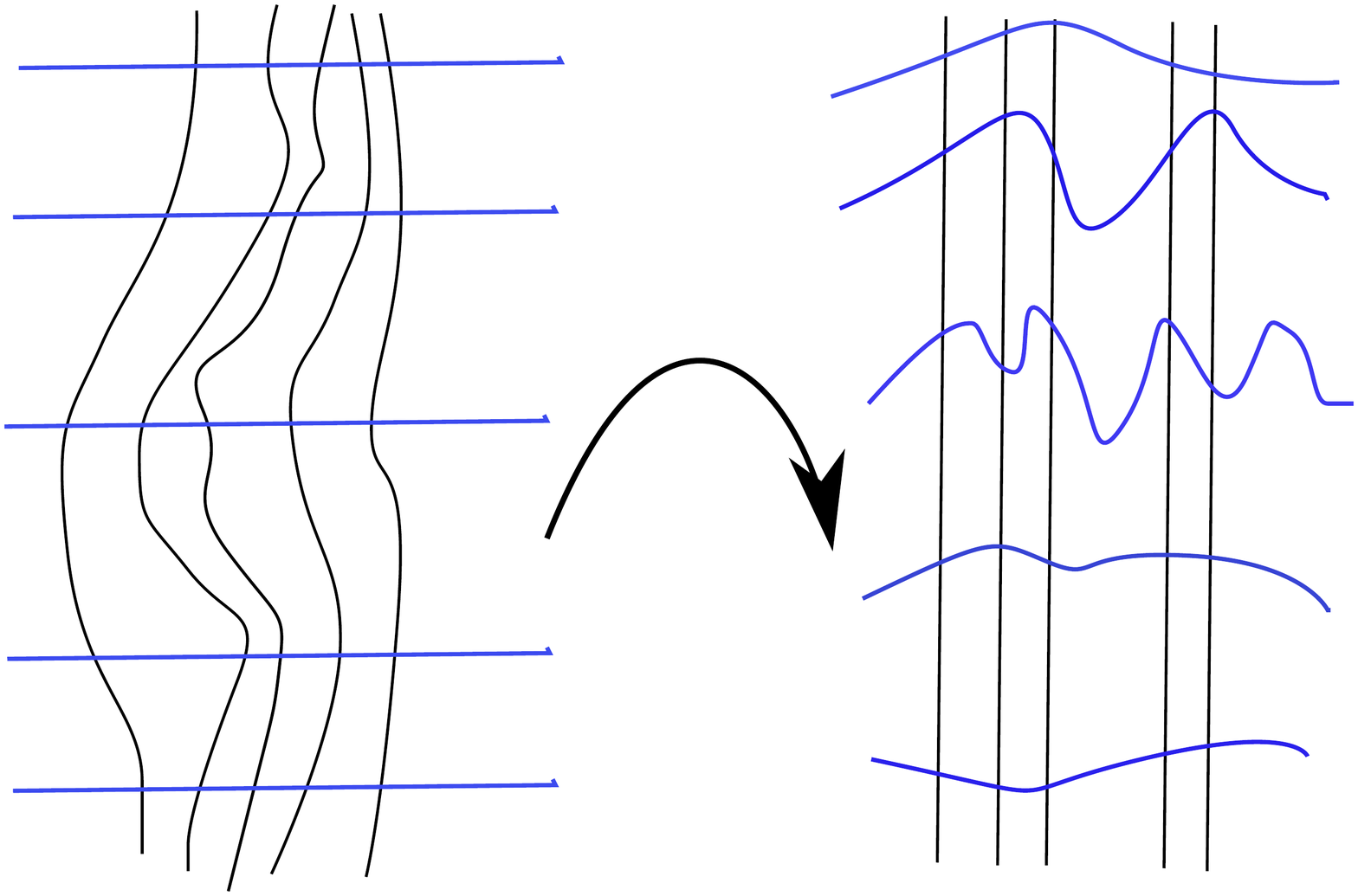}
\end{center}

We now use the preservation of the area. We fix $\theta_1<\theta_2$ in $\pi_1(\cV)$  and $c_1<c_2$ in $\pi_2(\cU)$ such that the domain $D_1$ delimited by the horizontals $\pi_2^{-1}( \{ c_1\})$, $\pi_2^{-1}( \{ c_2\})$, the graph of $c\in[c_1, c_2]\mapsto \zeta_{\theta_1}(c)$ and the graph of $c\in[c_1, c_2]\mapsto \zeta_{\theta_2}(c)$ is contained in $\cU$. Because $\Phi$ is a symplectic homeomorphism, $\Phi$ preserves the area and so $D_1$ and $\Phi(D_1)$ have the same area. Observe that  $\Phi(D_1)$ is the domain delimited by the graphs of $\eta_{c_1}$, $\eta_{c_2}$ and the verticals $\pi_1^{-1}(\{\theta_1\})$ and $\pi_1^{-1}(\{ \theta_2\})$.

This can be written
$$  \int_{c_1}^{c_2}\Big( \big(\theta_2+\frac{\partial v}{\partial c} (\theta_2,c)\big)-\big(\theta_1+\frac{\partial v}{\partial c} (\theta_1,c)\big)\Big) dc=\int_{\theta_1}^{\theta_2}\Big(\big(c_2+\frac{\partial u}{\partial \theta} (\theta,c_2)\big) - \big(c_1+\frac{\partial u}{\partial \theta} (\theta,c_1)\big)\Big)d\theta .$$
It follows that 
\begin{multline*}
u(\theta_2,c_2)-u(\theta_1,c_2)-u(\theta_2,c_1)+u(\theta_1,c_1) = \\
=v(\theta_2,c_2) -v(\theta_2,c_1) -v(\theta_1,c_2) +v(\theta_1,c_1).
\end{multline*}

Evaluating for $\theta_1=0$ we find
$$u(\theta_2,c_2)-u(\theta_2,c_1) = v(\theta_2,c_2) -v(\theta_2,c_1) -v(0,c_2) +v(0,c_1).$$
Finally, as $v$ admits a continuous partial derivative with respect to $c$, we conclude that 
$\frac{\partial u}{\partial c}(\theta, c) = \frac{\partial v}{\partial c}(\theta, c)-\frac{\partial v}{\partial c}(0, c)$ exists and is continuous. Hence $u$ is $C^1$.  Moreover, every map $\theta\mapsto \theta+\frac{\partial u}{\partial c}(\theta, c)= \zeta_c(\theta)-\frac{\partial v}{\partial c}(0, c)$ is a homeomorphism onto its image and we have established the first implication.\\
\subsection{ Proof of the second implication}
We assume that there exists a $C^1$ map $u : \cw \to \R$ such that 
\begin{itemize}
\item $u(0,c) = 0$ for all $c\in I$ where $I=(c_-, c_+)$ or $I=\R$,
\item $\eta_c(\theta) = c+ \frac{\partial u}{\partial \theta}(\theta,c)$ for all $(\theta, c) \in \cw$,
\item for all $c\in I$, the map $\theta \mapsto \theta + \frac{\partial u}{\partial c}(\theta , c)$ is a injective.
\end{itemize}
Then we can define a unique homeomorphism $\Phi$  by  $$\Phi\Big(\theta + \frac{\partial u}{\partial c}(\theta , c), c\Big)=\Big(\theta, c+ \frac{\partial u}{\partial \theta}(\theta,c)\Big).$$
Let $v:\R^2\rightarrow \R_+$ be the $C^\infty$ function with support in $B(0, 1)$ defined  by $v(\theta, c) = a \exp\big((1-\|(\theta, c)\|)^{-2}\big)$ for $(\theta, c) \in B(0,1)$ and where $a$ is    such that $\int  v=1$. We denote by $v_\varepsilon$ the function $v_\varepsilon (x)=\frac{1}{\varepsilon^2}v(\frac{x}{\varepsilon})$. Then we define for every $\varepsilon>0$.
$$U_\varepsilon(\theta, c)=(u*v_\varepsilon) (\theta, c),$$
where we recall the formula for the convolution
$$u*v(x)= \int u(x-y)v(y)dy.$$
Note that if we fix a compact into $\cw$, $U_\varepsilon$ is well defined on it for $\varepsilon$ small enough.\\
 Then when $\varepsilon$ tends to $0$, the functions $U_\varepsilon$ tend  to $U$ in the $C^1$ compact-open  topology. Moreover, when $\cU=\A$, $U_\varepsilon$ is  $1$-periodic in $\theta$  and smooth.\\
Observe that for every $\theta$, the function $c\mapsto c+\frac{\partial u}{\partial \theta}(\theta,c)$ is  increasing.
 We deduce that  the convolution $c\mapsto c+\frac{\partial U_\varepsilon}{\partial \theta}(\theta,c)$ is a $C^\infty$ diffeomorphism as it is a mean of $C^\infty$ diffeomorphisms  thanks to Lemma \ref{Philippe}. Finally, the maps $F_\varepsilon: (\theta, c)\mapsto \big(\theta,  c+\frac{\partial U_\varepsilon}{\partial \theta}(\theta, c)\big)$ define  $C^\infty$ foliations that  converge to the initial foliation $F_0: (\theta, c)\mapsto \big(\theta,  c+\frac{\partial u}{\partial \theta}(\theta, c)\big)$ for the $C^0$ compact-open topology when $\varepsilon$ tends to $0$.\\ 
Observe that the  $h_c$'s are   assumed to be increasing. We deduce  that the  maps $G_\varepsilon: (\theta, c)\mapsto (\theta+\frac{\partial U_\varepsilon}{\partial c}(\theta,c), c)$ are $C^\infty$ diffeomorphisms of $\A$ that    converge for the $C^0$  compact-open topology   to  $G_0:(\theta, c)\mapsto (\theta+\frac{\partial u}{\partial c}(\theta, c), c)$.\\
Finally, the $\ch_\varepsilon= F_\varepsilon\circ G_\varepsilon^{-1}$  are $C^\infty$ diffeomorphisms of $\A$ that converge   for the $C^0$  compact-open topology  to $F_0\circ G_0^{-1}=\Phi$.

This exactly means that $\Phi$ is a symplectic homeomorphism. 
\begin{lemma}\label{Philippe}
Let $f : \R\to \R$ be a non-negative, non-trivial, smooth, integrable and even function such that $ f' \leq 0$ on $[0,+\infty)$. Then if $g : \R \to \R$ is increasing, $f*g$ is an increasing $C^\infty$ diffeomorphism.
\end{lemma}
\begin{proof}
As $f$ is even, $f'$ is odd. Just notice that 
$$(f*g)'(x) = \int_\R f'(y)g(x-y) dy = \int_0^{+\infty} f'(y) \big(g(x-y)-g(x+y)\big) dy.$$
The result follows as $g(x-y)-g(x+y) <0$ and $f'(y) \leq 0$ and does not vanish everywhere.
\end{proof}

\section{Proof  of Proposition \ref{PLip} and Corollary \ref{bolle}}\label{Sfollip}
\subsection{Proof  of Proposition \ref{PLip}: biLipschitz foliations with $C^1$ generating functions are straightenable}

 Let $u:\A\rightarrow\R$ be the $C^1$ generating function of a continuous foliation of $\A$ into graphs. 
 
 We recall the following result that is due to Minguzzi, \cite{Min2014}.
 \begin{thm*}[Minguzzi]Let $\Omega$ be an open subset of $\R^2$ and let $f\in C^1(\Omega, \R)$. Then the following conditions are equivalent:
 \begin{enumerate}
 \item for every $x$, the partial derivative $\frac{\partial f}{\partial x}(x,\cdot)$ is locally Lipschitz, locally uniformly with respect to $x$;
 \item for every $y$, the partial derivative $\frac{\partial f}{\partial y}(\cdot,y)$ is locally Lipschitz, locally uniformly with respect to $y$.
 \end{enumerate}
 If they hold true, then on a subset $E\subset\Omega$ with full Lebesgue measure in $\Omega$, $\frac{\partial^2f}{\partial x\partial y}$ and $\frac{\partial^2f}{\partial y\partial x}$ exist and are equal.
 \end{thm*}
\subsubsection{Proof of the first implication}
 We assume that the invariant foliation is   $K$-Lipschitz on a compact $\ck =\{ (\theta, \eta_c(\theta)); \theta\in \T, c\in [a, b]\}$, which means 
\begin{equation}\label{Efollip}
\forall \theta \in \T, \forall c_1, c_2\in  [a,b], \quad \frac{|c_1-c_2|}{K}\leq \left| \eta_{c_1}(\theta)-\eta_{c_2}(\theta)\right|\leq K|c_1-c_2|.
\end{equation} 
As $\eta_c(\theta)=c+\frac{\partial u}{\partial \theta}(\theta, c)$, this means that $\frac{\partial u}{\partial \theta}(\theta, .)$ is locally Lipschitz, locally uniformly with respect to $\Theta$. Hence, by Minguzzi theorem, for every $c\in (a, b)$, $\frac{\partial u}{\partial c}(\cdot , c)$ is locally Lipschitz, locally uniformly with respect to $c$ and at almost $(\theta, c)\in \T\times (a, b)$, we have
$\frac{\partial^2u}{\partial c\partial \theta}$ and $\frac{\partial^2u}{\partial \theta\partial c}$ exist and are equal and uniformly bounded.\\
Hence $h_c=Id_\T+ \frac{\partial u}{\partial c}(\cdot ,c )$ is locally uniformly Lipschitz and 
because of Equation (\ref{Efollip}), we have Lebesgue almost everywhere
\begin{equation}\label{Ehlip} \frac{\partial^2u}{\partial \theta\partial c}(\theta, c)= \frac{\partial^2u}{\partial c\partial \theta}(\theta, c)=\frac{\partial \eta}{\partial c}(\theta_0, c_0)-1\in \big[ -1+\frac{1}{K}, -1+K\big]=[k_-, k_+].\end{equation}
This implies that $h_c(\theta)=\theta +\frac{\partial u}{\partial c}(\theta, c)$ defines a $(k_-, k_+)$-biLipschitz   homeomorphism of $\T$ for almost every $c\in [a, b]$ and then for all $c\in [a, b]$ by continuity. By Theorem \ref{TC0arn}, we deduce that $u$ is the generating function of an exact symplectic homeomorphism   $\Phi:\A\rightarrow \A$  that  maps the invariant foliation onto the standard one.
\subsubsection{  Proof of the second implication} We assume that  the map $u$ is $C^1$ with $\frac{\partial u}{\partial \theta}$ locally Lipschitz continuous and $\frac{\partial u}{\partial c}$ uniformly Lipschitz in the variable  $\theta$ on any compact set of $c$'s and there exists two constants $k_+>k_->-1$ such that $\frac{\partial^2 u}{\partial \theta\partial c}(\theta,c)\in [k_-, k_+]$ almost everywhere. Because $\frac{\partial u}{\partial c}$ is uniformly Lipschitz in the variable  $\theta$ on any compact set of $c$'s, we can apply Minguzzi theorem and write another time
 Equation (\ref{Ehlip}) which implies that the foliation is biLipschitz.
 \subsection{Proof  of  Corollary \ref{bolle}}
Let $k\geq 1$ and $r\mapsto f_r$ be a $C^k$-foliation in graphs and let $u$ be its generating function. As $k\geq 1$, the foliation is Lipschitz when restricted to every compact set. Hence we can use Proposition \ref{PLip}.   In this case,  $u$ is the generating function of an exact symplectic homeomorphism   $\Phi:\A\rightarrow \A$  that  maps the standard foliation onto the invariant one  and we have
 $$\Phi\Big(\theta + \frac{\partial u}{\partial c}(\theta , c), c\Big)=\Big(\theta, c+ \frac{\partial u}{\partial \theta}(\theta,c)\Big).$$
 Moreover, $u$ is   $C^k$ hence $F_0(\theta, c)=\big(\theta, c+\frac{\partial u}{\partial \theta}(\theta, c)\big)$ defines a $C^{k-1}$ homeomorphism that is locally biLipschitz, hence a $C^{k-1}$ diffeomorphism. \\
Also  $h_c(\theta)=\theta+\frac{\partial u}{\partial c}(\theta, c)$  is $C^{k-1}$ in $(\theta, c)$. Observe that every $h_c$ is a biLipschiz homeomorphism that is $C^{k-1}$, hence $G_0: (\theta, c)\mapsto (h_c(\theta), c)$ is also a $C^{k-1}$ diffeomorphism  and then $\Phi=F_0\circ G_0^{-1}$ is a $C^{k-1}$ symplectic diffeomorphism (where a $C^0$- diffeomorphism is an homeomorphism).

 \section{More results on symplectic homeomorphisms that are $C^0$-integrable}\label{shomeomC0}
 
\subsection{Proof of Proposition \ref{Psymfolarn}}
 Let $f:\A \to \A$ be an exact  symplectic homeomorphism.  We assume  that $f$ has an invariant foliation $\cF$ into $C^0$ graphs that is  symplectically homeomorphic (by $\Phi^{-1} : \A \to \A$) to the standard foliation $\cF_0=\Phi^{-1}(\cF)$. Then the standard foliation is invariant by the exact symplectic homeomorphism $g=\Phi^{-1}\circ f\circ \Phi$. Hence we have
 $$g(\theta, r)=(g_1(\theta ,r), r).$$
 As $g$ is area preserving, for every $\theta\in [0, 1]$ and every $r_1<r_2$, the area of $[0, \theta]\times [r_1, r_2]$ is equal to the area of $g\big( [0, \theta]\times [r_1, r_2]\big)$, i.e.
 $$\theta(r_2-r_1)=\int_{r_1}^{r_2} \big(g_1(\theta, r)-g_1(0, r)\big)dr.$$
 Dividing by $r_2-r_1$ and taking the limit  when $r_2$ tends to $r_1$, we obtain
 $$g_1(\theta, r_1)=\theta+g(0, r_1).$$
 This proves the proposition for $\rho=g_1(0, \cdot)$.
 
 \subsection{Proof  of   Corollary \ref{Corota}}
 The  {\em   if} part is obvious by Proposition \ref{Psymfolarn}.

Let us prove the {\em  only if} part, that is we assume $f$ is $C^0$-integrable with the Dynamics on each leaf conjugated to a rotation. We denote by    $u : \A \to \R$ the map given by  theorem  \ref{Tgeneder} and that enjoys the properties of Theorem \ref{Tdiffrat}.  Hence $h_c:\theta \mapsto \theta+\frac{\partial u}{\partial c}(\theta,c)$ is a semi-conjugation between the projected Dynamics $g_c: \theta\mapsto \pi_1\circ f\big(\theta, c+\frac{\partial u}{\partial \theta}(\theta,c)\big)$ and the rotation $R_{\rho(c)}$ of $\T$ and even is a conjugation when $\rho(c)$ is rational.\\
If $\rho(c)$ is irrational, it follows from the hypothesis that $g_c$ is conjugated to a rotation. As the Dynamics is minimal, there is up to constants a unique (semi)-conjugacy and and then $h_c$ is a true conjugation. We then conclude by using Theorem \ref{TC0arn}.

\subsection{Proof of   Corollary  \ref{corLip} }

\subsubsection{Arnol’d-Liouville coordinates for $f$}

 Let $f:\A\rightarrow\A$ be a symplectic twist diffeomorphism that is Lipschitz-integrable with generating function $u$ of its invariant foliation.\\
  By Theorem \ref{Tgeneder} and Proposition \ref{PLip}, $u$ is the generating function of an exact symplectic homeomorphism   $\Phi:\A\rightarrow \A$  that  maps the standard foliation onto the invariant  one and for every compact subset $\ck\subset \A$, there exists two constant $k_+>k_->-1$ such that $\frac{\partial^2 u}{\partial \theta\partial c}\in[k_-, k_+]$  Lebesgue almost everywhere in $\ck$.
  
  By Proposition \ref{Psymfolarn}, we have $$\forall (x, c)\in \A,\quad  \Phi^{-1}\circ f\circ \Phi(x, c)=(x+\rho(c), c);$$
where $\rho:\R\rightarrow \R$ is continuous. Moreover, because of the twist condition, $\rho$ is an increasing homeomorphism of $\R$. 

\subsubsection{Proof that $\rho:\R\rightarrow \R$ is a biLipschitz homeomorphism}

\begin{proposition}\label{rho}
Assume that the $C^1$ symplectic twist diffeomorphism  $f:\A\rightarrow \A$ has an invariant  locally Lipschitz continuous  foliation into graphs $c\in\R\mapsto  \eta_c\in C^{0}(\T, \R)$. Then the map $\rho: c\in\R\mapsto \rho(c)$ is a locally biLipschitz homeomorphism.
\end{proposition}
We will use the following
\begin{lemma}\label{lemme rho}
Let $f,g : \R \to \R$ be lifts of homeomorphisms of $\T$ that preserve orientation (implying $f(\cdot +1) = f(\cdot) +1$ and $g(\cdot +1) = g(\cdot) +1$). Assume that 
\begin{itemize}
\item either $f$ or $g$ is conjugated to a translation $t_\alpha : x \mapsto x+\alpha$ by a homeomorphism $h$ that is a lift of a homeomorphism of $\T$ that preserves orientation;
\item $h$ and $h^{-1}$ are $K$-Lipschitz.
\end{itemize}
Then
\begin{enumerate}
\item If there exists $d>0$ such that $f<g+d$, then
 $\rho(f)\leq \rho(g)+Kd$.
 \item If there exists $d>0$ such that $f+d <g$ 
 then $\rho(f) +\frac{d}{K}\leq \rho(g)$.
 \end{enumerate}
\end{lemma}
\begin{proof}
Let us say that $h\circ g \circ h^{-1} = t_\alpha$, hence $\rho(g) = \alpha$ (the proof when $f$ is conjugated to a translation is the same). 
\begin{enumerate}
\item 
By hypothesis, $ f\circ h^{-1}< g\circ h^{-1} +d $. Using that $h$ is increasing and  $K$-Lipschitz, it follows that for all $ x\in \R$,
$$ \quad h\circ f\circ h^{-1}(x)<h(g\circ h^{-1}(x) +d)< h\circ g\circ h^{-1}(x)+Kd=x+\alpha+Kd.$$
Finally, as $\rho(f) = \rho(  h\circ f\circ h^{-1})$, we conclude that 
$$\rho(f)\leq \alpha+Kd = \rho(g)+Kd.$$
\item By hypothesis, $ f\circ h^{-1}+d< g\circ h^{-1}  $. Using that $h$ is increasing, it follows that
$$\forall x\in \R, \quad h( f\circ h^{-1}(x)+d)<h\circ g\circ h^{-1}(x)=x+\alpha.$$
Because $h^{-1}$ is $K$-Lipschitz and increasing, observe that
\begin{multline*}
d=h^{-1}(h(f\circ h^{-1}(x)+d))-h^{-1}(h\circ f\circ h^{-1}(x))\\
\leq K\left( h(f\circ h^{-1}(x)+d)-h\circ f\circ h^{-1}(x)\right).
\end{multline*}
Then
$$h\circ f\circ h^{-1}(x)\leq h(f\circ h^{-1}(x)+d)-\frac{d}{K}<x+\alpha -\frac{d}{K};$$
hence $\rho(f)+\frac{d}{K}\leq \rho(g)$.

\end{enumerate}
\end{proof}

\begin{proof}[Proof of Proposition \ref{rho}]
The proof is now a direct application of the previous Lemma. Indeed, we have seen that when the foliation is $K$-Lipschitz, if $c$ varies in a compact  set $\mathcal K$, the Dynamics $g_c$ are all conjugated to rotations. We have moreover proven there exists a constant $\widetilde K$ such that the conjugating functions $h_c$ may be chosen equi-biLipschitz (for $c\in \mathcal K$).

 We denote the minimum and maximum torsions on $\ck$ by
$$b_{\rm min}=\min_{x\in\ck}\frac{\partial f_1}{\partial \theta}(x)\quad\text{and}\quad b_{\rm max}=\max_{x\in\ck}\frac{\partial f_1}{\partial \theta}(x).$$
For $c_1<c_2$ in $[a, b]$, we have

$$\tilde g_{c_2}(\theta)-\tilde g_{c_1}(\theta)=F_1\big(\theta, \eta_{c_2}(\theta)\big)-F_1\big(\theta, \eta_{c_1}(\theta)\big)$$

and so 
$$\tilde g_{c_2}(\theta)-\tilde g_{c_1}(\theta)\in \Big[b_{\rm min}\big(\eta_{c_2}(\theta)-\eta_{c_1}(\theta)\big), b_{\rm max}\big(\eta_{c_2}(\theta)-\eta_{c_1}(\theta)\big)\Big]$$

and

$$\tilde g_{c_2}(\theta)-\tilde g_{c_1}(\theta)\in \Big[\frac{b_{\rm min}}{K}(c_2-c_1), K.{b_{\rm max}}.(c_2-c_1)\Big].$$

We deduce from  Lemma \ref{lemme rho} that
$${K.\tilde K}.b_{\rm max}(c_2-c_1)\geq \rho(g_{c_2})-\rho(g_{c_1})\geq \frac{b_{\rm min}}{K.\tilde K}(c_2-c_1).$$
\end{proof}
\subsubsection{Proof of the $C^1$ regularity} 
Here we prove that $\Phi:\A\rightarrow \A$    is $C^1$ in the $\theta$ variable, that the invariant foliation is a $C^1$ lamination    and that the Dynamics restricted to every leaf is $C^1$ conjugated to a rotation.

Let us fix $c$. Then $h_c=Id_\T+\frac{\partial u}{\partial c}(., c)$ is a biLipschitz homeomorphism of $\T$ by Proposition \ref{PLip}. Then Corollary 4 of \cite{Arna1}  tells us that $\eta_c$ is in fact $C^1$ (and the two Green bundles coincide along its graphs) and that $h_c$ is a $C^1$ diffeomorphism.\\
Hence all the points of $\A$ are recurrent. Moreover, as the two Green bundles are equal everywhere, they are continuous. Because they coincide with the tangent space to the foliation, the foliation is a $C^1$ lamination. This is equivalent to the continuity (in the two variables) of $\frac{\partial^2u}{\partial   \theta^2}$.\\
As $\Phi(\Theta, c)=\big(h_c^{-1}(\Theta), \eta_c\circ h_c^{-1}(\Theta)\big)$,  we deduce that $\Phi$ is $C^1$ in the $\theta$-direction.

\remk we don't know if $\frac{\partial^2u}{\partial \theta\partial c}$ is continuous, and then if  $c\mapsto h_c$ is continuous for the $C^1$ topology.

\section{A strange foliation}\label{sstrange}
We consider the  foliation of $\A$ by the graphs of $\eta_c(\theta) =  c+ \varepsilon(c)\cos(2\pi \theta)$ where $\varepsilon$ is a contraction ($k$-Lipschitz with $k<1$) that is not everywhere differentiable. It is a biLipschitz foliation with smooth leaves. Observe that the generating function of this foliation is given by 
$$u(\theta, c)=\frac{\varepsilon(c)}{2\pi}\sin(2\pi \theta).$$

 \subsection {Proof of Corollaries \ref{Corstrangefolia} and \ref{Lipnotstraighten}}
 As $u$ is not $C^1$, we deduce from  Theorems \ref{Tgeneder} and \ref{TC0arn} that this foliation cannot be globally straightenable by a symplectic homeomorphism and also that it cannot be invariant by a symplectic twist diffeomorphism.
 
 Let us prove the local part of Corollary \ref{Lipnotstraighten}. Then  we work in $\cw=[\alpha, \beta]\times I$ and we define $U$ by
 $$U(\theta, c)=u(\theta, c)-\frac{u(\beta, c)-u(\alpha, c)}{\beta-\alpha}.$$
 Then $U$ is not $C^1$ and we deduce from Theorem \ref{TC0arn} that the local foliation is not straightenable via a symplectic homeomorphism.

\subsection{An exact symplectic twist map that leaves the strange foliation invariant} Let us prove however that this foliation, for a simple choice of $\varepsilon$, can be  invariant by a certain $C^1$ exact symplectic twist map. 
 \begin{defi} An exact symplectic homeomorphism $f:\A\rightarrow \A$ has the  {\em weak  twist property} if
when $F=(F_1, F_2):\R^2\rightarrow \R^2$ is any lift of $f$, for any $\tilde \theta\in \R$, the map $r\in \R\mapsto F_1(\tilde \theta, r)\in \R$ is an increasing homeomorphism from $\R$ onto $\R$.
 \end{defi}
 
 Let us now assume that $\varepsilon$ is a $C^2$ function away from $c=0$ and that at $0$ it has a left and a right derivatives up to order $2$. For the sake of simplicity, let us assume also that $\varepsilon(0)=0$ so that $\T\times \{0\}$ is a leaf of the foliation and that $\varepsilon$ restricted to $[0,+\infty)$ (resp. $(-\infty , 0]$) is the restriction of a $C^2$ periodic function.

 The proof of Theorem \ref{bolle} gives us two $C^1$ functions 
 $$\Phi^\pm : (\theta,r ) \mapsto \big(h^\pm(\theta,r) , \eta( h^\pm(\theta,r), r) \big)$$
 where $\Phi^+$ is a $C^1$ exact symplectic diffeomorphism of $\A^+ = \T \times [0,+\infty)$ to itself (up to the boundary) and $\Phi_-$ is a $C^1$ exact symplectic diffeomorphism of $\A^- = \T \times (-\infty, 0]$ to itself (up to the boundary). Note that here $\Phi^+$ and $\Phi^-$ do not coincide on $\T\times \{0\}$ explaining why the foliation is not straightenable.
 
 Let $\rho : \R \to \R$ be an increasing, $C^1$ homeomorphism such that $\rho(0) = \rho'(0) = 0$. We denote by $f_\rho : (\theta,r) \mapsto (\theta+\rho(r) , r)$. The function $f = \Phi^\pm \circ f_\rho \circ (\Phi^\pm)^{-1}$ is well defined on $\A$, it is the identity on $\T \times \{0\}$. It is clearly an area preserving homeomorphism that is $C^1$ away from $\T\times \{0\}$. 
 
 If $r> 0$ and $\theta\in \T$, let us set $(\Theta,R) = \Phi^+(\theta,r)   $. Then one finds that
 \begin{multline*}
 D f(\Theta,R) = D\Phi^+(\theta + \rho(r),r) \cdot Df_\rho ( \theta,r)\cdot D\Phi^+(\theta,r)^{-1} \\
 =      D\Phi^+(\theta + \rho(r),r)  \cdot \begin{pmatrix}1&\rho'(r)\\0&1
\end{pmatrix} \cdot D\Phi^+(\theta,r)^{-1}
 \end{multline*}
 
 It follows from the properties on $\Phi^+$ and $\rho(0)=\rho'(0) = 0$ that as $R\to 0$, $Df(\Theta,R)$ uniformly converges to the identity. As the same holds for $R<0$, we deduce that $f$ is in fact $C^1$ with a differential on $\T \times \{0\}$ being identity. 
 
 It is left to chose $\rho$ in such a way that the obtained map is a twist map. We construct it on $[0,+\infty)$. The twist condition we aim at is: for every $\Theta\in\R$, the map $r\mapsto h^+\big(\left(h^{+}_r\right)^{-1}(\Theta)+\rho(r),r\big)$ is an increasing homeomorphism of $\R$.
 
 After computation, if we denote $h^+(\theta_r, r)=\Theta$, the derivative of the above function is the following for $r>0$ (the inequality is our goal): 
 
 \begin{multline*}
 \frac{\partial h^+}{\partial r}(\theta_r+\rho(r),r)-\left( \frac{\partial h^+}{\partial \theta}(\theta_r,r)\right)^{-1}\frac{\partial h^+}{\partial \theta}(\theta_r+\rho(r),r)\frac{\partial h^+}{\partial r}(\theta_r,r) \\
 + \frac{\partial h^+}{\partial \theta}(\theta_r+\rho(r),r)\rho'(r)>0.
 \end{multline*}
 The first line above is smaller in absolute value than $M_1\rho(r)$ where (recall that by hypothesis, all the functions at play are continuous periodic hence bounded)
  $$M_1 = \left \|\frac{\partial^2 h^+}{\partial r\partial \theta}\right \|_\infty +\left  \|\left(\frac{\partial h^+}{\partial \theta}\right)^{-1}\right \|_\infty . \left  \|\frac{\partial^2 h^+}{\partial \theta^2}\right \|_\infty  \left \|\frac{\partial h^+}{\partial r}\right \|_\infty .$$

 
 On the other hand, the second line is greater than $M_2 \rho'(r)$ where we set $ M_2 =  \min \frac{\partial h^+}{\partial \theta} >0$. 
   If $\rho(t) = t^2 e^{Mt}$ with $M= 2M_1/M_2$, then we have $\rho(0)=\rho'(0)=0$ and $\rho'(t)>\frac{M_1}{M_2}\rho(t)$ that implies the twist condition.

 \appendix
 
 \section{A  foliation by graphs that is the inverse image of the standard foliation by a symplectic map but not by a symplectic  homeomorphism}\label{AppA}
 We will use two special functions
 \begin{itemize}
 \item $\gamma:\T\rightarrow \R$ a $C^\infty$ function such that $\gamma'_{[\frac{1}{2}-\varepsilon, \frac{1}{2}+\varepsilon]}=-1$ and 
 $\gamma'_{\T\backslash [\frac{1}{2}-\varepsilon, \frac{1}{2}+\varepsilon]}>-1$;
 \item $\zeta:\R\rightarrow \R$ a $C^\infty$ function that is increasing, such that $\zeta'(0)=1$ and $\zeta'_{\R\backslash\{ 0\}}<1$ with $\displaystyle{\lim_{\pm\infty}\zeta'=\frac{1}{2}}$.
 \end{itemize}
 The function $u(\theta, c)=\zeta(c)\gamma(\theta)$ defines the foliation in graphs of $$\eta_c=c+\frac{\partial u}{\partial\theta}=c+\zeta(c)\gamma'.$$
 The derivative with respect to $c$ of $\eta_c(\theta)$ is then $\frac{\partial \eta_c}{\partial c}(\theta)=1+\zeta'(c)\gamma'(\theta)$ that is non negative, vanishes only for $(\theta, c)\in [\frac{1}{2}-\varepsilon, \frac{1}{2}+\varepsilon]\times \{0\}$ and is larger that $\frac{1}{3}$ close to $\pm\infty$. Hence every map $c\in\R\mapsto \eta_c(\theta)\in\R$ is a homeomorphism and we have indeed a $C^0$ foliation.\\
 Let us introduce $h_c(\theta)=\theta+\frac{\partial u}{\partial c}(\theta)=\theta+\gamma(\theta)\zeta'(c)$. Its derivative is $1+\zeta'(c)\gamma'(\theta)$  that is non negative and vanishes only if $(\theta, c)\in [\frac{1}{2}-\varepsilon, \frac{1}{2}+\varepsilon]\times \{0\}$. Hence $h_0$ is not a homeomorphism but all the other $h_c$ are homeomorphisms. \\
 We deduce from Theorem \ref{TC0arn} that this foliation is not symplectically homeomorphic to the standard one.
 
 We will now prove that the map defined by $H\big(\theta, \eta_c(\theta)\big)=(h_c(\theta), c)$ is a symplectic map, i.e. the limit (for the $C^0$ topology) of 
a sequence  of symplectic diffeomorphisms.\\
Let $\gamma_n:\T\rightarrow \R$ be a sequence of $C^\infty$ maps that converges to $\gamma$ in $C^1$ topology and satisfies $\gamma_n'>-1$. Let $(\zeta_n)$ be a sequence of $C^\infty$ diffeomorphisms of $\R$ that $C^1$ converges to $\zeta$ and satisfies $\zeta_n'<1$. We introduce $u_n(\theta, c)=\gamma_n(\theta)\zeta_n(c)$. Then $\eta_{c,n}(\theta)=c+\zeta_n(c)\gamma_n'(\theta)$ defines a smooth foliation, $h_{c,n}(\theta)=\theta+\gamma_n(\theta)\zeta_n'(c)$ is a smooth diffeomorphism of $\T$ and
$$K_n(\theta, c)=\left( \left( h_{c,n}\right)^{-1}(\theta), \eta_{c,n}\left(\big( h_{c,n}\right)^{-1}(\theta)\big)\right)$$
is a symplectic smooth diffeomorphism that maps the standard foliation  to the foliations by the graphs of $\left( \eta_{c,n}\right)_{c\in\R}$.\\
If $H_n=K_n^{-1}$, observe that $H_n=G_n\circ F_n^{-1}$ where 
\begin{itemize}
\item $F_n(\theta, c)=\big(\theta, c+\frac{\partial u_n}{\partial \theta}(\theta, c)\big)$ converges uniformly to $F(\theta, c)=\big(\theta, c+\frac{\partial u}{\partial \theta}(\theta, c)\big)$;
\item $G_n(\theta, c)=(\theta+\frac{\partial u_n}{\partial c}(\theta, c), c)$ converges uniformly to $G(\theta, c)=(\theta+\frac{\partial u}{\partial c}(\theta, c), c)$.
\end{itemize}
Finally, $H_n=G_n\circ F_n^{-1}$ converges uniformly to $H=G\circ F^{-1}$

 \section{Green bundles}\label{ssGreenb} Here we recall the theory of Green bundles. More details or proofs can be found in \cite{Arna3,Arna1}.
We fix a lift $F$ of an symplectic twist diffeomorphism $f$.  
\begin{notas}\label{Nota}{\rm
\begin{enumerate}
\item[$\bullet$] $V(x)=\{ 0\}\times\R\subset T_x\R^2$ and for $k\not= 0$, we have $G_k(x)=DF^k(F^{-k}x)V(f^{-k}x)$;
\item[$\bullet$] the slope of $G_k$ (when defined) is denoted by $s_k$: $$G_k(x)=\{ (\delta\theta, s_k(x)\delta\theta); \ \ \delta\theta\in\R\};$$
\item[$\bullet$] if $\gamma$ is a real Lipschitz function defined on $\T$ or $\R$, then 
$$\gamma'_+(x)=\limsup_{\substack{y,z\rightarrow x\\ y\not=z}}\frac{\gamma(y)-\gamma(z)}{y-z}\quad{\rm and}\quad \gamma'_-(t)=\liminf_{\substack{y,z\rightarrow x\\ y\not=z}}\frac{\gamma(y)-\gamma(z)}{y-z}.$$
\end{enumerate}}
\end{notas}
Then  
\begin{enumerate}
\item if the orbit of $x\in\R^2$ is minimizing,  we have 
$$\forall n\geq 1,\quad s_{-n}(x)<s_{-n-1}(x)<s_{n+1}(x)<s_n(x);$$
\item in this case,   the two {\em Green bundles} at $x$ are $G_+(x), G_-(x)\subset T_x(\R^2)$ with slopes $s_-$, $s_+$ where $\displaystyle{s_+(x)=\lim_{n\rightarrow +\infty}s_n(x)}$ and $\displaystyle{ s_-(x)=\lim_{n\rightarrow +\infty}s_{-n}(x)}$;
\item the two Green bundles  are invariant under $Df$: $Df(G_\pm)=G_\pm\circ f$;
\item we have $s_+\geq s_-$;
\item the map $s_-$ is lower semi-continuous and the map  $s_+$ is upper semi-continuous;
\item hence $\{ G_-=G_+\}$ is a  $G_\delta$ subset of the set of points whose orbit is minimizing (this last set  is a closed set) and   $s_-=s_+$ is continuous at every point of this set.
\end{enumerate}

Let us focus on the case of an invariant curve that is the graph of $\gamma$. Then we have
\begin{propos}\label{PGreensand}
Assume that the graph of $\gamma\in C^0(\T, \R)$ is invariant by $F$. Then the orbit of any point contained in the graph of $\gamma$ is minimizing and we have
$$\forall \theta\in\T,\quad s_-\big(\theta, \gamma(\theta)\big)\leq\gamma'_-(\theta)\leq \gamma'_+(\theta)\leq s_+\big(\theta, \gamma(\theta)\big).$$
\end{propos}

\begin{propos}\label{Pdyncrit} {\bf (Dynamical criterion)} Assume that $x$ has its orbit that is minimizing and that is contained in some strip $\R\times[-K,K]$ (for example $x$ is in some invariant graph) and that $v\in T_x\R^2\backslash\{ 0\}$. Then
\begin{enumerate}
\item[$\bullet$] if $\displaystyle{\liminf_{n\rightarrow +\infty} |D(\pi\circ F^n)(x)v|<+\infty}$, then $v\in G_-(x)$;
\item[$\bullet$] if $\displaystyle{\liminf_{n\rightarrow +\infty} |D(\pi\circ F^{-n})(x)v|<+\infty}$, then $v\in G_+(x)$.
\end{enumerate}
\end{propos} 
In particular, if the Dynamics restricted to some invariant graph is totally periodic, then along this graph we have $G_-=G_+$ and the graph is $C^1$. The $C^1$ property can also  be proved  by using the implicit functions theorem.

\bibliographystyle{amsplain}

\end{document}